\documentclass[12pt]{article}

\usepackage{mathptmx}
\usepackage{fancyhdr, amssymb, amsmath, amsthm, enumerate, enumitem, lastpage}     
\usepackage{mathrsfs}
\usepackage{tikz, tikz-cd}
\usepackage{array}
\usepackage{fmtcount}
\usepackage{multirow}
\usepackage{mathtools}
\usepackage{hyperref}
\usepackage{cite, url}
\newtheorem{thm}{Theorem}[section]
\newcounter{introthm}[section]
\newenvironment{introthm}[1][]{ \refstepcounter{introthm}\par\medskip
   \noindent \textbf{Theorem \Alph{introthm}: #1} \itshape\rmfamily}{}
\newenvironment{introprop}[1][]{ \refstepcounter{introthm}\par\medskip
   \noindent \textbf{Proposition \Alph{introthm}: #1} \itshape\rmfamily}{}

\newenvironment{remark}[1][]{ \refstepcounter{thm}\par\medskip
   \noindent \textit{Remark \arabic{section}.\arabic{thm}: #1} \rmfamily}{}
\newenvironment{question}[1][]{ \refstepcounter{thm}\par\medskip
   \noindent \textbf{Question \arabic{section}.\arabic{thm}: #1} \rmfamily}{}
\newenvironment{obs}[1][]{ \refstepcounter{thm}\par\medskip
   \noindent \textbf{Observation \arabic{section}.\arabic{thm}: #1} \rmfamily}{}
\newenvironment{const}[1][]{ \refstepcounter{thm}\par\medskip
   \noindent \textbf{Construction \arabic{section}.\arabic{thm}: #1} \rmfamily}{}
\newenvironment{emple}[1][]{ \refstepcounter{thm}\par\medskip
   \noindent \textbf{Example \arabic{section}.\arabic{thm}: #1} \rmfamily}{}
\newtheorem{lemma}[thm]{Lemma}
\newtheorem{prop}[thm]{Proposition}
\newtheorem{defn}[thm]{Defintion}
\newtheorem{cor}[thm]{Corollary}
\newtheorem{claim}{Claim}[section]
\newcommand{\kpdeg}{[\kappa \, : \, \kappa^p]}

\newcommand{\trc}[3]{\mathrm{trace}_{#1/#2} \left( #3 \right)}
\newcommand{\tray}[2]{\mathrm{Tr} \left( #1, #2 \right)}

\newcommand{\wt}[1]{\widetilde{#1}}
\newcommand{\mc}[1]{\mathcal{#1}}
\newcommand{\ehk}[1]{e_{\mathrm{HK}}\left(#1\right)}
\newcommand{\ehkideal}[2]{e_{\mathrm{HK}}\left(#1, #2\right)}
\newcommand{\len}[2]{\lambda_{#1}\left(#2\right)}
\newcommand{\height}[1]{\mathrm{ht} \left(#1\right)}
\newcommand{\depth}[2]{\mathrm{depth}_{#1} \left(#2\right)}
\newcommand{\col}[3]{\left( #1 \, :_{#3} \, #2 \right)}
\newcommand{\hgy}[2]{\mathbf{H}_{#1}\!\!\left(#2\right)}

\newcommand{\knl}[1]{\mathrm{Ker} \left(#1\right)}

\newcommand{\minh}[1]{\mathrm{minh} \left(#1\right)}
\newcommand{\spec}[1]{\mathrm{Spec}\left(#1\right)}
\newcommand{\im}[1]{\mathrm{Im} \left(#1\right)}

\newcommand{\supp}[2]{\mathrm{Supp}_{#1}{\left(#2\right)}}
\newcommand{\mf}[1]{\mathfrak{#1}}
\newcommand{\ann}[2]{\mathrm{Ann}_{#1}\left(#2\right)}

\newcommand{\fbp}[1]{\left[ #1\right]}
\newcommand{\unideal}[1]{\left(\underline{#1}\right)}
\newcommand{\AR}[3]{\mathrm{AR}_{#3}\left(#1 \, | \, #2\right)}
\newcommand{\ol}[1]{\overline{#1}}
\newcommand{\indt}{\hspace{.25in}}
\newcommand{\rank}[2]{\mathrm{rank}_{#1}\left(#2\right)}

\newcommand{\disc}[2]{D_{#1} \left(#2\right)}
\newcommand{\mtiny}[1]{\scriptscriptstyle{#1}}
\newcommand{\olper}[1]{\ol{#1}_{\mtiny{\unideal{f+\epsilon}}}}
\newcommand{\per}[1]{#1_{\mtiny{\unideal{f+\epsilon}}}}
\newcommand{\ul}[1]{\underline{#1}}
\newcommand{\ulper}[1]{\ul{#1}_{\mtiny{\unideal{f+\epsilon}}}}
\newcommand{\hm}[3]{\mathrm{Hom}_{#1}\left(#2, \, #3\right)}
\newcommand{\nilrad}[1]{\mathrm{nilrad}\left(#1\right)}
\title{}
\author{}
\date{}
\fancypagestyle{plain}{
  \fancyhf{}
  \fancyhead[R]{Perturbations in Noetherian Local Rings}
  \fancyfoot[R]{\thepage\  / \pageref{LastPage}}
}
\begin{document}
\maketitle
\begin{center}
\text{\Large $\mf{m}$-adic Perturbations in Noetherian Local Rings} \\
\small{\textbf{Nick Cox-Steib} \,\, \textit{noc3md@mail.missouri.edu} \\ University of Missouri}
\end{center}
\begin{abstract}
We develop new methods to study $\mf{m}$-adic stability in an arbitrary Noetherian local ring.  These techniques are used to prove results about the behavior of Hilbert-Samuel and Hilbert-Kunz multiplicities under fine $\mf{m}$-adic perturbations.
\end{abstract}
\flushleft{}
\section{Introduction}
\indt The topic of this paper involves comparing a fixed ideal, $I = (f_1, \dots, f_c) = \unideal{f},$ inside a Noetherian local ring, $(R, \mf{m}_R),$ with ideals of the form $(f_1 + \epsilon_1, f_2 + \epsilon_2, \dots, f_c + \epsilon_c) = \unideal{f+\epsilon},$ where $\epsilon_1, \dots, \epsilon_c$ are in some large power of the maximal ideal of $R$.  When it is in the interest of clarity and space and the ideal $I$ is understood, we often use the notation
\begin{align*}
	\ol{M} := \dfrac{M}{IM} = \dfrac{M}{\unideal{f}M}, \,\,\,\,\,\, \text{  and  } \,\,\,\,\,\, \olper{M} := \dfrac{M}{\unideal{f+\epsilon}M}
\end{align*}
where $M$ is an $R$-module.  We will refer to the ideals $\unideal{f+\epsilon}$, as $\mf{m}_R$-adic {\it perturbations} of $I,$ and, more generally, we may also speak of the quotients $\olper{M}$ as {\it perturbations} of $\ol{M}.$ At the most basic level, we want to understand the relationship between $I$ and its perturbations.  \\
\indt Recent interest in these problems stems from their relationship to the \textit{deformation} of properties.  This relationship has been formalized in a recent paper by De Stefani and Smirnov with the concept of $\mf{m}$\textit{-adic stability} \cite{AIstab}. A property, $\mc{P},$ of Noetherian local rings is $\mf{m}$\textit{-adically stable} if it is stable under sufficiently small perturbations of ideals generated by a regular element --- that is, for every regular element $f$ in a Noetherian local ring $(R, \mf{m})$ such that $R/(f)$ has property $\mc{P},$ there is an $N \in \mathbb{N}$ such that $R/(f+\epsilon)$ also has property $\mc{P}$ for all $\epsilon \in \mf{m}^N.$  \\
\indt Recall that a property of Noetherian local rings, $\mc{P},$ is said to \textit{deform} if whenever $x \in R$ is a regular element in $(R, \mf{m}_R)$ and $R/(x)$ has property $\mc{P},$ then $R$ has $\mc{P}.$  Regular, Cohen-Macaulay and Gorenstein are all properties that are known to deform in this sense, and whether a given property does or does not deform is often of fundamental interest.  De Stefani and Smirnov show, for properties satisfying mild conditions, $\mf{m}$-adically stability implies deformation (Theorem A of \cite{AIstab}).   \\
\indt The fluctuation of numerical measures of singularity under $\mf{m}$-adic perturbations has also generated interest.  In their foundational paper \cite{ST} Srinivas and Trivedi investigate how the Hilbert function changes when the defining equations of a quotient are perturbed.  Using an ingenious Artin-Rees argument they show, in general, that the Hilbert function can only decrease when the perturbations are sufficiently small (lemma 3, \cite{ST}).  This investigation has been continued in a recent paper of Ma, Quy, and Smirnov, where they showed that if $I$ is generated by a filter-regular sequence a stronger result holds --- in this case, the Hilbert function is actually an invariant of sufficiently small $\mf{m}$-adic perturbations of $I$ \cite{Mafilter}.  In a 2018 paper Polstra and Smirnov asked these kinds of questions about $F$-invariants in Cohen-Macaulay rings.  They prove, when $I \subset R$ is a parameter ideal in a Cohen-Macaulay, $F$-finite local ring, that the Hilbert-Kunz multiplicity exhibits a remarkable kind of $\mf{m}$-adic continuity with respect to small perturbations of the generators of $I$  (see theorem 4.3 of \cite{polstra_smirnov_2018} and corollary 3.7 of \cite{AIstab} as well as Theorem \ref{IllyaThomasHK} below for a statement). They prove a similar result about $F$-signature in $F$-finite Gorenstien rings (theorem 3.11 \cite{polstra_smirnov_2018}).  This research is continued in \cite{AIstab} and \cite{polstra2020fpurity}.\\
\indt The idea of perturbing the equations defining a singularity is quite natural, and from this perspective, interest in $\mf{m}$-adic perturbations is not new to commutative algebra.  For example in \cite{hironaka1965equivalence}, building on work of Samuel \cite{samuel1956algebricite}, Hironaka considers perturbations inside of a power series ring $S = k[[x_1, \dots, x_n]].$  His work shows that if $S/I$ is a reduced equidimensional isolated singularity, then for any $J \subset S$ obtained as a small enough perturbation of the generators of $I$, satisfying the additional conditions that $\height{J} = \height{I}$ and that $J \subset S$ also defines a reduced equidimensional isolated singularity, there is an automorphism $\sigma: S \to S$ taking $I$ to $J$, $\sigma(I) = J.$  This work was continued by Cutkosky and Srinivasan in \cite{10.2307/2375013, cutkosky1997equivalence}.  The behavior of complexes under perturbations of the differentials is investigated in \cite{E1974} and \cite{eisenbud2005finiteness}, with interesting applications to regular sequences and depth.  Related questions are explored in \cite{huneke1997height} and \cite{eisenbud1977remarks}. \\
\subsection{Summary of Main Results}
\indt All of the results described above impose strict restrictions on the ideals being perturbed or the kinds of perturbations allowed, and most make strong assumptions about the properties of the ambient ring  \cite{samuel1956algebricite, hironaka1965equivalence, 10.2307/2375013, cutkosky1997equivalence, Mafilter, polstra_smirnov_2018}.  The goal of this paper is to introduce an approach to studying $\mf{m}$-adic perturbations of ideals that works in an arbitrary Noetherian local ring.  As an application, we use these methods to establish new results about the $\mf{m}$-adic behavior of Hilbert-Samuel and Hilbert-Kunz multiplicity. \\
\indt In the statements below, $\hgy{1}{I; M}$ denotes the Koszul homology of $I$ with coefficients in the $R$-module $M$ (note that, in general, the computation of $\hgy{1}{I; M}$ does not commute with localization --- see section \ref{notation} for details).  
\begin{introthm} \label{introMults}
Suppose that $(R, \mf{m}_R)$ is an equicharactersitic Noetherian local ring, and that $I \subset R$ is an ideal.  Set $\ol{R} = R/I,$ and assume that $a := \dim \ol{R} \ge 1.$ 
\begin{itemize} 
\item[(A.1)] (Theorem \ref{HSper}) Let $M$ be a finite $R$-module.  Assume that 
\begin{align*}
	\dim_R \hgy{1}{I; M} = \dim \dfrac{R}{\ann{R}{\hgy{1}{I; M}}} < a \, .
\end{align*}
Then, for any $\mf{m}_R$-primary ideal $J \supset I,$ there is a $T \in \mathbb{N}$ such that, for all minimal generators $(f_1, \dots, f_c) = I$ and all $\epsilon_1, \dots, \epsilon_c \in \mf{m}_R^T$, 
	\begin{align*}
		e\left(\, J \ol{R}, \, \ol{M}\right) = e\left(\, J \olper{R}, \,  \olper{M}\right) \, .
	\end{align*}	
\item[(A.2)] (Theorem \ref{HKThm}) Suppose that $(R, \mf{m}_R)$ has positive characteristic $p > 0$ and is $F$-finite.  Assume the following three conditions hold:
\begin{itemize}
\item[(i)] $\widehat{R}/I\widehat{R}$ is equidimensional ;
\item[(ii)] $\dim_{R/I} \hgy{1}{I; R} < a$ ;
\item[(iii)]  $\dim_{\widehat{R}/I\widehat{R}} \nilrad{\widehat{R}/I\widehat{R}} < a$.
\end{itemize}
Let $J \subset R$ be an $\mf{m}_R$-primary ideal. \\
\indt Then, for any $\delta > 0,$ there is a $T \in \mathbb{N}$ such that for all minimal generators $(f_1, \dots, f_c) = I$ and all $\epsilon_1, \dots, \epsilon_c \in \mf{m}_R^T$
\begin{align*}
	\Biggl|\ehk{ J \ol{R}} - \ehk{ J\olper{R}} \Biggr| < \delta \, .
\end{align*}  
\end{itemize}
\end{introthm}
\indt  When $\dim \ol{R} = 0,$ the corresponding results are trivial --- see corollary \ref{nakcordim0}.  Note that while (A.1) is a weaker conclusion than theorem 3.7 of \cite{Mafilter}, the result applies to considerably more general perturbations.  The second statement in theorem A, (A.2), is a substantial generalization of the Hilbert-Kunz continuity result of \cite{polstra_smirnov_2018} to local rings that are not Cohen-Macaulay.  The proofs of these statements are found in section \ref{AppsMult}.  A.2 has the following interesting corollary:
\begin{introthm} (Theorem \ref{HKlimit}) Let $(R, \mf{m}_R, \kappa)$ be a reduced, $F$-finite local ring, and suppose that $I \subset R$ is an ideal.  Assume that the conditions of theorem (A.2) are satisfied, and let $$(f_1, \dots, f_c) = I$$ be any minimal generating set.  Then,
	\begin{align*}
		\lim_{n_1 \to \infty, \dots, n_c \to \infty} \ehk{R[f_1^{1/n_1}, \dots, f_c^{1/n_c}]} = \ehk{R/I}.
	\end{align*}
\end{introthm}
\indt The proofs of (A.1) and (A.2) use the methods developed in sections \ref{SectionBasic} and \ref{TechTools}.  In section \ref{SectionBasic} we introduce the basic setup for the theory.  Not every ideal behaves well when it's minimal generators are perturbed and we require that certain technical conditions are satisfied --- in section \ref{techcon} we define and explore these technical conditions.  With the setup established, the main technical lemmas of the paper are proven in \ref{TechTools}.  With an eye towards future applications, we have made an effort to prove these results in maximum generality.  In section \ref{refine} these technical tools are put into a form more concrete and suitable for the scope of this paper.  For example, we prove the following:
\begin{introthm}
Suppose that $(R, \mf{m}_R)$ and $(A, \mf{m}_A)$ are complete Noetherian local rings, that $I \subset R$ is an ideal, and there is a commuting diagram of local rings of the following form
$$D :\begin{tikzcd}[row sep=large, column sep=large]    
			& R \arrow[d, twoheadrightarrow] \\
		 A \arrow[r, hookrightarrow, "finite", swap] \arrow[ru, hookrightarrow]	 & R/I
 \end{tikzcd}$$
 where the extension $A \xhookrightarrow{} \ol{R}$ is module finite.  Note that, in this case, $I + \mf{m}_AR \subset R$ is $\mf{m}_R$-primary. 
\begin{itemize}
\item[(C.1)] (Lemma \ref{mgenslemma})  Let $T = T(D)$ be any element of $\mathbb{N}$ such that $\mf{m}_R^T \subset \mf{m}_R \left(I + \mf{m}_A R \right).$ \\
\indt Then, for any $(f_1, \dots, f_c) = I$ and any $ \epsilon_1, \dots, \epsilon_c \in \mf{m}_R^T,$ \\
\begin{itemize}
	\item[(i)] the composition $A \to R \to \olper{R}$ is module finite ; \\
    \item[(ii)] Suppose $M$ is a finite $R$-module, and let $\ul{m} = m_1, \dots, m_{\mtiny{\mu_A(\ol{M})}} \in M$ be elements whose images in $\ol{M}$ are minimal $A$-module generators for $\ol{M}.$  Then $\ul{m}$ is mapped to an $A$-minimal generating set in $\olper{M}.$
\end{itemize}
\item[(C.2)] (Theorem \ref{freeper}) Assume that $\depth{\mf{m}_A}{A} \ge 1$, and that
$M$ is a finite $R$-module such that $d \hgy{1}{I; M} = 0$ for some $d \in A$ which is a nzd in $A.$ \\
\indt Let $m_1, \dots, m_n \in M$ be any elements such that the image of 
\begin{align*}
	N := Am_1 + \dots + Am_n
\end{align*}
is a free $A$-module of rank $n$ in $\ol{M}$. \\
\indt Then, there is a $T = T(D, M, N),$ such that for all $\epsilon_1, \dots, \epsilon_c \in \mf{m}_R^T$ and all minimal generators $(f_1, \dots, f_c) = I$ the image of $N$ in $\olper{M}$ is a free $A$-module of rank $n$.
\end{itemize}
\end{introthm}  
\indt  The diagram and conditions appearing in the statements of C are covered in detail in section \ref{SectionBasic}. With the tools in hand, we turn to the proofs of (A.1) and (A.2).  Section \ref{HSsection} covers (A.1), which is a relatively straightforward application of standard techniques.  The proof of (A.2), developed in section \ref{HKsection}, requires more work.  To extend the argument of \cite{polstra_smirnov_2018} we need to use the trace map and discriminant outside of their classical scope.  The details of this, which are 'known to the experts,' are worked out in \ref{traceSec}.  In the process, we prove results about the behavior of trace maps, and generically \'etale extensions under small perturbations (see section \ref{traceSec} for definitions and details).
\begin{introprop}
Suppose that $I \subset R$ is an ideal, and that $(R, \mf{m}_R), (A, \mf{m}_A)$, and $\ol{R} = R/I$ are complete Noetherian local rings related as in the diagram, $D,$ appearing in theorem C. \\
\indt Assume that $\dim \ol{R} \ge 1,$ that $A$ is a normal local domain, and that there is a nonzero $d \in A$ such that $d \hgy{1}{I; R} = 0.$ \\
\begin{itemize}
\item[(D.1)] (Lemma \ref{traceper} and Corollary \ref{Discper}) For any $Q \in \mathbb{N}$ there is a $T \gg 0,$ such that for all minimal generators, $(f_1, \dots, f_c) = I,$ and all $\epsilon_1, \dots, \epsilon_c \in \mf{m}_R^T,$ 
\begin{align*}
	\trc{\ol{R}}{A}{r} - \trc{\olper{R}}{A}{r} \in \mf{m}_A^{Q}
\end{align*}
for every $r \in R.$  And consequently, 
\begin{align*}
	\disc{A}{\ol{R}} - \disc{A}{\olper{R}} \in \mf{m}_A^{Q}.
\end{align*}
\item[(D.2)] (Proposition \ref{genetaleper}) Suppose that $A \xhookrightarrow{} \ol{R}$ is generically \'etale. \\
\indt Then, there is a $T \gg 0$ such that for all minimal generators $(f_1, \dots, f_c) = I$ and any $\epsilon_1, \dots, \epsilon_c \in \mf{m}_R^T,$ the composition $A \xhookrightarrow{} R \to \olper{R}$ is a generically \'etale module finite extension.
\end{itemize}
\end{introprop}
\indt Once this foundation is in place, the basic elements of the argument in \cite{polstra_smirnov_2018} carry over, and we prove (A.2) in section \ref{HKsection}.  The paper concludes with some interesting examples and some suggestions for future research.
\subsection{Acknowledgments} 
\indt The present work is based on the author's thesis.  It would not have been possible without the patience and support of many people.  In particular, the author would like to thank his adviser, Ian Aberbach, for being so generous with his time, guidance and friendship.  The author would also like to thank Thomas Polstra for carefully wading through a draft of this paper and graciously suggesting several valuable improvements.  Above all else, the author is grateful for the unconditional love and support of his wife and family. 
\section{$\mf{m}$-adic Perturbations} \label{SectionBasic}
\subsection{Notation} \label{notation}
\indt Throughout $(R, \mf{m}_R)$ is a Noetherian local ring.  We use $\mu_R(M)$ to denote the minimal number of generators of a finite $R$-module $M.$  For ideals $I \subset R$, $$\minh{R/I} := \{ \mf{p} \in \spec{R} \, | \, \mf{p} \supset I \, \, \text { is minimal,  and } \,\, \dim R/I = \dim R/\mf{p} \}.$$
\begin{defn} An $\mf{m}_R$-adic {\bf perturbation} of $I \subset R,$  of {\bf order} $T$, is an ideal $(f_1 + \epsilon_1, \dots, f_c + \epsilon_c) = \unideal{f+\epsilon} \subset R,$ where $(f_1, \dots, f_c) = I$ are minimal generators, and $\epsilon_1, \dots, \epsilon_c \in \mf{m}_R ^T$. \\
A perturbation is {\bf small} if its order, $T$, is 'large' in some sense which is clear from the context.
\end{defn}
\indt The Artin-Rees lemma makes multiple appearances throughout the paper.  If $J \subset R$ is an ideal and $M' \subset M$ are finite $R$-modules, then $\AR{J}{M' \subset M}{R}$ denotes the corresponding Artin-Rees number --- i.e. letting $t = \AR{J}{M' \subset M}{R}$ we have $J^{n + t} M \cap M' = J^n \left(J^t \cap M'\right) \subset J^nM',$ for all $n \ge 0.$  For elements $g_1, \dots, g_m \subset R$ in a ring $R,$ and $M$ an $R$-module, $$\mathbf{K}\left(g_1, \dots, g_m ; M\right) = \mathbf{K}\left(\unideal{g} ; M\right)$$ is the corresponding Koszul complex on $M.$  The differentials of $\mathbf{K} \left( \unideal{g} ; M \right)$ are denoted by $\partial_{i, \unideal{g}}^M.$  If $J \subset R$ is an ideal in a local ring, $R$, the Koszul complex of $J$ with coefficients in $M$ is $\mathbf{K}\left( J ; M \right) = \mathbf{K} \left( \unideal{g} ; M \right),$ where $J = \unideal{g}$ are minimal generators --- recall that, since we are working in a local ring, this complex, and it's homology $\hgy{*}{J; M},$ are well defined up to isomorphism (see, e.g., the discussion on the bottom of pg 52 of \cite{bh1998}). \\
\subsection{Perturbations, First Results and Diagram $D$}
	\indt Our treatment begins with a fundamental lemma from \cite{ST}.
	\begin{lemma} \label{naklemma}
Suppose that $I$ and $J$ are ideals in a local Noetherian ring $(R, \mf{m}_R).$ \\
Then, for any $\epsilon_1, \dots, \epsilon_c \in \mf{m}_R\left( I + J \right)$ and any generators, $(f_1, \dots, f_c) = I,$ there is an equality
\begin{align*}
	I + J = \unideal{f + \epsilon} + J
\end{align*}  
where $\unideal{f+\epsilon} = \left( f_1 + \epsilon_1, \dots, f_c + \epsilon_c \right).$
\end{lemma}
\begin{proof}
Notice that
	$$I + J \subset \unideal{f+\epsilon} + J + \mf{m}_R\left( I + J \right) \subset I + J + \mf{m}_R \left(I + J\right) = I + J$$
	and now the result follows from Nakayama's lemma.
\end{proof}
\indt Taking $I$ to be an $\mf{m}_R$-primary ideal and setting $J = 0$ in lemma \ref{naklemma}, we recover a well known consequence of Nakayama's lemma. \\
\indt For future reference, we record this conclusion in a different, more suggestive, form.
\begin{cor} \label{nakcordim0}
Suppose $(R, \mf{m}_R)$ is a Noetherian local ring, and $I \subset R$ is an $\mf{m}_R$-primary ideal.  Let $T \in \mathbb{N}$ be such that 
\begin{align*}
	\mf{m}_R^T \subset \mf{m}_R I \, .
\end{align*}
Then, for any $(f_1, \dots, f_c) = I ,$ and any $R$-module $M$, there is an equality
\begin{align*}
	\ol{M} = \olper{M}
\end{align*}
for all $\epsilon_1, \dots, \epsilon_c \in \mf{m}_R^T.$ 
\end{cor}
\indt Perturbing ideals that are not $\mf{m}_R$-primary can be considerably more complicated (see the examples in section \ref{examples}).  To deal with this we need some additional structure to work with.  \\
\indt Suppose that $(R, \mf{m}_R)$ and $(A, \mf{m}_A)$ are Noetherian local rings, that $I \subset R$ is an ideal, and there is a commuting diagram of local rings of the following form
$$D :\begin{tikzcd}[row sep=large, column sep=large]    
			& R \arrow[d, twoheadrightarrow] \\
		 A \arrow[r, hookrightarrow, "finite", swap] \arrow[ru, hookrightarrow]	 & R/I
 \end{tikzcd}$$
 where the extension $A \xhookrightarrow{} \ol{R}$ is module finite.  Throughout the paper, when we speak of rings $R, A$ and $R/I$ being 'related as in diagram $D$,' we are referring to this diagram.  \\
\indt This configuration serves two functions.  We show in lemma \ref{mprimaryepsilons} that the setup of diagram $D$ gives us a way to parameterize the $\mf{m}_R$-adic perturbations of $I$ in $R$ using $\mf{m}_A$-primary ideals from $A$.  The diagram's other function is that it provides a useful way to re-frame the problem --- our goal is to relate the $A$-module structures of $\ol{R}$ and it's perturbations, $\olper{R}.$ \\
\indt This diagram appears naturally in many contexts, and is already present in the arguments of \cite{polstra_smirnov_2018}.  When $R$ is equal characteristic, the Cohen structure theorem produces, for every choice of coefficient field and system of parameters on $R/I$, a diagram of this form, with $A$ a power series ring (see remark \ref{mixedchar} for the mixed characteristic case).   
\begin{const} \label{sopDiagram} 
 Suppose that $I$ is an ideal in an equal characteristic complete Noetherian local ring $(R, \mf{m}_R, \kappa).$  \\
 \indt Fix a coefficient field for $R$, $\kappa \xhookrightarrow{} R,$ and recall that the composition $\kappa \xhookrightarrow{} R \twoheadrightarrow R/I$ is a coefficient field for $\ol{R} : = R/I.$  Given any full system of parameters on the quotient, $(x_1, \dots, x_a) \subset \mf{m}_{\ol{R}} \subset \ol{R},$ there is a module finite extension
	\begin{align*}
		A := \kappa[[x_1, \dots, x_a]] \xhookrightarrow{} \ol{R} \, .
	\end{align*}   
	If we choose preimages of the $x_1, \dots, x_a \in \ol{R}$ under the quotient map $R \twoheadrightarrow \ol{R},$ they will be part of a system of parameters up in $R,$ and there is an injection $A \xhookrightarrow{} R,$  which is a lift of the module finite extension $A \xhookrightarrow{} \ol{R}.$ \\
 \indt In this way, the choice of a system of parameters on $\ol{R}$ produces a commutative diagram of local rings of the form mentioned above:
	$$D :\begin{tikzcd}[row sep=large, column sep=large]    
			& R \arrow[d, twoheadrightarrow] \\
		 A \arrow[r, hookrightarrow, "finite", swap] \arrow[ru, hookrightarrow]	 & R/I
 \end{tikzcd}$$
\end{const}
 	\indt Diagram $D$ also appears in equimultiplicity theory (see \cite{I2016equimult}, \cite{I2019semicont} and \cite{lipman1982equimultiplicity}).  In that context, one studies the behavior of invariants parameterized by the spectrum of $A$.  Coincidentally, the $\mf{m}_A$-primary ideals in $A$ serve a similar function for us.\\
\begin{lemma} \label{mprimaryepsilons}
 	Suppose $(R, \mf{m}_R)$ and $(A, \mf{m}_A)$ are Noetherian, local rings, that $I \subset R$ is an ideal minimally generated by $f_1, \dots, f_c \in R$, and that $A$, $R$ and $\ol{R}=R/I$ are related as in the diagram $D$.
  Let $J \subset A$ be an $\mf{m}_A$-primary ideal.  Then
\begin{itemize}
  \item[(i)] $I + JR \subset R$ is $\mf{m}_R$-primary ;
  \item[(ii)] Given any $\epsilon_1, \dots, \epsilon_c \in \mf{m}_R \left(I + JR \right),$ there are minimal generators $(f'_1, \dots, f'_c) = I,$ and $\epsilon'_1, \dots, \epsilon'_c \in JR$ such that
\begin{align*}
	(f_1 + \epsilon_1, \dots, f_c+\epsilon_c) = (f'_1+\epsilon'_1, \dots, f'_c + \epsilon_c) \, .	
\end{align*}
\end{itemize} 	
\end{lemma}
\begin{proof}
For (i), simply note that $A \xhookrightarrow{} \ol{R}$ is module finite, and 
\begin{align*}
	\dfrac{R}{I + JR} \cong  \dfrac{\ol{R}}{J\ol{R}}
\end{align*}
which has finite length. \\
\indt To establish (ii), let $(f_1, \dots, f_c) = I$ be minimal generators, and suppose $\epsilon_1, \dots, \epsilon_c \in \mf{m}_R\left(I + JR \right).$ \\
Observe that
\begin{align*}
\epsilon_i \in \mf{m}_R \left(I + JR\right) \subset \mf{m}_R I + JR
\end{align*}
and thus we may write
\begin{align*}
 \epsilon_i = h_i + \epsilon'_i
\end{align*}
with $h_i \in \mf{m}_R I$ and $\epsilon'_i \in JR$ for each $i.$  \\
Then the elements, $f'_i = f_i + h_i,$ are minimal generators for $I = \unideal{f}$ by Nakayama's lemma, and the desired equality 
\begin{align*}
	\unideal{f+\epsilon} = \unideal{f'+\epsilon'}
\end{align*} 
holds by construction.
\end{proof}
\begin{lemma} \label{mgenslemma}
	Suppose $(R, \mf{m}_R)$ and $(A, \mf{m}_A)$ are complete Noetherian, local rings, that $I \subset R$ is an ideal, and $D$ is a commuting diagram of local rings and local ring maps as above. \\
Let $T = T(D)$ be any element of $\mathbb{N}$ such that $\mf{m}_R^T \subset \mf{m}_R \left(I + \mf{m}_A R \right).$  Then, for any $(f_1, \dots, f_c) = I$ and any $ \epsilon_1, \dots, \epsilon_c \in \mf{m}_R^T,$ \\
\begin{itemize}
	\item[(i)] the composition $A \to R \to \olper{R}$ is module finite; \\
    \item[(ii)] if $M$ is a finite $R$-module, any collection of elements, $m_1, \dots, m_{\mtiny{\mu_A(\ol{M})}} \in M$ whose images are minimal $A$-module generators in $\ol{M}$, is mapped to a minimal generating set for $\olper{M}$ as an $A$-module. 
\end{itemize}
\end{lemma}
\begin{proof}  First note that (i) follows from (ii), so it suffices to prove (ii). \\
\indt According to lemma \ref{naklemma}, given any $\epsilon_i \in \mf{m}_R^T \subset \mf{m}_R \left(I + \mf{m}_A R \right)$ and any minimal generators $\unideal{f} = I,$ there is an equality $$I + \mf{m}_AR = \unideal{f+\epsilon} + \mf{m}_A R.$$ Thus, for any f.g. $R$-module $M$ we have an isomorphism of $A$-modules $$\ol{M} \otimes_A A/\mf{m}_A \cong \dfrac{M}{(I+\mf{m}_AR)M} = \dfrac{M}{(\unideal{f+\epsilon}+\mf{m}_A R)M} \cong \olper{M} \otimes_A A/\mf{m}_A.$$ \\
\indt Further, it is easily seen that each $\olper{M}$ is separated in the $\mf{m}_A$-adic topology,
$$\bigcap_{k\ge 1} \left(\mf{m}_A ^k \olper{M} \right) \subset \bigcap_{k\ge 1} \left(\mf{m}_R^k\olper{M}\right) = 0.$$ 
 \indt Applying theorem 8.4 of \cite{matsumura1989commutative}, we see that if $m_1, \dots, m_{\mtiny{\mu_A(\ol{M})}} \in M$ are any elements mapping to a basis in each of the modules $$\olper{M} \otimes_A A/\mf{m}_A \cong \ol{M} \otimes_A A/\mf{m}_A$$ their images will be generators in every $\olper{M}.$  They are minimal generators by Nakayama's lemma.\\
\end{proof}
\begin{remark} \label{injmap} The techniques developed in section \ref{TechTools} will allow us to show, if $I$ satisfies some additional assumptions, these induced local maps, $A \to \olper{R},$ are actually injections for sufficiently large $T.$ 
\end{remark} \\
\vspace{.1in}
\indt Lemma \ref{mgenslemma} says that the same elements in a finite $R$-module, $m_1, \dots, m_{\mtiny{\mu_A(\ol{M})}} \in M,$ are simultaneous minimal $A$-module generators for all sufficiently close perturbations of $\ol{M}$.  In other words, as $A$-modules, these close perturbations, $\olper{M},$ are all quotients of the same {\it finite $A$-submodule of $M$,} $$N := Am_1+\dots + Am_{\mtiny{\mu_A(\ol{M})}} \subset M.$$ \\
\indt Motivated by this observation, we will investigate how the images of a given finite $A$-submodule $N \subset M$ in these quotients, $\olper{M},$ are related to one another.  Of particular interest for the applications in section \ref{AppsMult} is the case when $N \subset M$ maps to a free $A$-submodule in $\ol{M}$ (see theorem \ref{freeper}).   \\
\indt The following result and its corollaries are used later, and serve to illustrate these ideas.\\
\begin{lemma} \label{anntors} Assume the setup of lemma \ref{mgenslemma}, and that $M$ is a finite $R$-module.  Let $N \subset M$ be a finite $A$-submodule of $M,$ and write $\ul{N} \subset \ol{M}$ for the image of $N$ under the map $M \twoheadrightarrow \ol{M}.$  \\
\indt Suppose that $$\alpha \in \ann{A}{\ol{M}/\ul{N}}, \,\,\, \text{ i.e. } \,\,\, \alpha \ol{M} \subset \ul{N}.$$ \\
Then, for any integer, $H \ge 1,$ any $\epsilon_1, \dots, \epsilon_c \in \mf{m}_R \left(I + \mf{m}_A^H R\right),$ and any $(f_1, \dots, f_c) = \unideal{f} = I,$
\begin{align*}
	\alpha^{\mtiny{\mu_A(\ol{M})}} \in \ann{A}{\olper{M}/\ulper{N}} + \mf{m}_A^H,
\end{align*}
where $\ulper{N} \subset \olper{M}$ is the image of $N$ in $\olper{M}$.
\end{lemma}
\begin{proof}
Recall that, for any such $\epsilon$'s and generators, $(f_1, \dots, f_c) = I,$ we have, by lemma \ref{naklemma}, 
	\begin{align*}
		I + \mf{m}_A^H R = \unideal{f+\epsilon} + \mf{m}_A^H R,
	\end{align*}	
	and, in particular, $$I \subset \unideal{f+\epsilon} + \mf{m}_A^H R.$$
	Now, $\alpha\ol{M} \subset \ul{N}$ means that $$\alpha M \subset IM + N,$$ and therefore
	\begin{align*}
		\alpha \left(\dfrac{M}{\unideal{f+\epsilon}M + N}\right) \subset \dfrac{IM + N + \unideal{f+\epsilon}M}{\unideal{f+\epsilon}M + N}	\subset \dfrac{ \unideal{f+\epsilon}M + \mf{m}_A^H M + N}{\unideal{f+\epsilon}M + N} = \mf{m}_A^H \left(\dfrac{M}{\unideal{f+\epsilon}M + N}\right)	\, .\\
	\end{align*}
	In other words, multiplication by $\alpha$ satisfies, 
	\begin{align*}
		\alpha\left(\olper{M}/\ulper{N}\right) \subset \mf{m}_A ^H\left(\olper{M}/\ulper{N}\right)
	\end{align*}
Now, according to lemma \ref{mgenslemma}, each $\olper{M}$ can be generated by the images of the same $\mu_A(\ol{M})$ elements as an $A$-module, and so the same is true of the quotients, $$\olper{M}/\ulper{N}.$$
Therefore, for any fixed $\unideal{f+\epsilon},$ we may apply the 'determinant trick' of \cite{HSintegral}, to conclude that there are $h_1, h_2, \dots, h_{\mtiny{\mu_A(\ol{M})}} \in A,$ with $h_i \in \mf{m}_A^{iH},$ such that multiplication by the element
	\begin{align*}
		\alpha^{\mtiny{\mu_A(\ol{M})}} + h_1 \alpha^{\mtiny{\mu_A(\ol{M})}-1} + \dots + h_{\mtiny{\mu_A(\ol{M})}-1} \alpha + h_{\mtiny{\mu_A(\ol{M})}} \in A
	\end{align*}
	annihilates $\olper{M}/\ulper{N},$
	i.e.
	\begin{align*}
		\alpha^{\mtiny{\mu_A(\ol{M})}} + h_1 \alpha^{\mtiny{\mu_A(\ol{M})}-1} + \dots + h_{\mtiny{\mu_A(\ol{M})}-1} \alpha + h_{\mtiny{\mu_A(\ol{M})}} \in \ann{A}{\olper{M}/\ulper{N}} \, .
	\end{align*}
This shows that $\alpha^{\mtiny{\mu_A(\ol{M})}}$ differs from an element in $\ann{A}{\olper{M}/\ulper{N}}$ by something in $\mf{m}_A^H,$ and the demonstration is complete.
\end{proof}
\begin{cor} \label{nonzeroann}
	Assume the same setup and notation as lemma \ref{anntors} and that $\alpha \in \ann{A}{\ol{M}/\ul{N}}$ is nonzero. If $Q \ge 1$ is such that $\alpha^{\mtiny{\mu_A(\ol{M})}} \not\in \mf{m}_A^Q,$ then 
	\begin{align*}
		\ann{A}{\olper{M}/\ulper{N}} \neq 0
	\end{align*}
	for every $\epsilon_1, \dots, \epsilon_c \in \mf{m}_R \left(I + \mf{m}_A^Q R\right),$ and any $(f_1, \dots, f_c) = \unideal{f} = I$.
	\end{cor}

\indt This already tells us something interesting about the $A$-module properties of small perturbations. \\
\begin{remark} \label{obsRank} \par 
\begin{itemize}
\item[(i)] Taking $N=0$ in corollary \ref{nonzeroann} we see that if $\ol{M}$ is torsion as an $A$-module, the same will be true of the $\olper{M}$ for perturbations of sufficiently high order. \\
\item[(ii)] Suppose that $A$ is a domain and that $m_1, \dots, m_n \in M$ are elements whose images in $\ol{M} = M/IM$ generate a free $A$-submodule of maximal rank, i.e.
	\begin{align*}
		A^{\oplus n} \cong A\ol{m}_1+\dots+A\ol{m}_n \subset \ol{M},
	\end{align*}
	with $\ol{M}/\left( A\ol{m}_1+\dots+A\ol{m}_n \right)$ $A$-torsion. \\
	Consider the finite $A$-submodule of $M$ spanned by these elements, 
	\begin{align*}
		N := Am_1+\dots+Am_n \subset M.
	\end{align*}
	Notice that $N$ must also be free over $A,$ and is taken isomorphically to it's image, $\ul{N},$ under the quotient map $M \twoheadrightarrow \ol{M}$ (see the proof of theorem \ref{freeper} for further discussion of this).  \\
	\indt Now, $\ul{N} \subset \ol{M}$ is a maximal rank free $A$-submodule, so there is a non-zero $\alpha \in A$ such that $\alpha\ol{M} \subset \ul{N}.$  If we let $Q$ be any integer $\ge 1$ such that $\alpha^{\mu_{A}(\ol{M})} \not\in \mf{m}_A^Q,$ then corollary \ref{nonzeroann} says that the image of $N \cong A^{\oplus n}$ in any $\olper{M}$ with the $\epsilon_1, \dots, \epsilon_c \in \mf{m}_A^Q$ {\it will be an $A$-submodule of $\olper{M}$ of maximal rank.}   
\end{itemize}
\end{remark}
\begin{cor} \label{rankle}
	Suppose $(R, \mf{m}_R)$ and $(A, \mf{m}_A)$ are complete Noetherian, local rings, that $A$ is a domain, $I \subset R$ is an ideal, $M$ is a finite $R$-module, and $A$, $R$, and $\ol{R}$ are related via diagram $D.$
 Then there is a $T = T(D, M)$ such that for all $(f_1, \dots, f_c) = I$ and all $\epsilon_1, \dots, \epsilon_c \in \mf{m}_R^T$
 \begin{align*}
 	\rank{A}{\olper{M}} \le \rank{A}{\ol{M}},
 \end{align*}
 where the rank of a torsion module is defined to be zero.
\end{cor}
	\begin{proof}
		Notice that, in the observation above, $I+\mf{m}_A^Q R$ is $\mf{m}_R$-primary, so there is a $T = T(D, M)$ such that $\mf{m}_R^T \subset \mf{m}_R(I + \mf{m}_A^Q R).$
	\end{proof}
\indt The inequality in corollary \ref{rankle} is not new. For example, it can be derived from lemma 3 of \cite{ST}, which also appears as lemma 3.2 of \cite{Mafilter}.
\subsection{Technical Conditions} \label{techcon}
\indt  With some additional assumptions, we will show, in section 3, that the inequality in corollary \ref{rankle} is an equality when the perturbations are small enough.  We need our data to satisfy some technical conditions: \\
	\begin{defn} \label{condition} Suppose that  $(R, \mf{m}_R)$ and $(A, \mf{m}_A)$ are Noetherian local rings, $I \subset R$ is an ideal, and $D$ is our standard commuting diagram of local ring maps,
 $$ D: \begin{tikzcd}[row sep=large, column sep=large]    
			& R \arrow[d, twoheadrightarrow] \\
		 A \arrow[r, hookrightarrow, "finite", swap] \arrow[ru, hookrightarrow]	 & R/I
 \end{tikzcd}$$
 Given the information of a finitely generated $R$-module, $M,$ together with a non-zero finitely generated $\mathit{A}$-{\it submodule,} $N \subset M,$ we say that a non-zero element $d \in A$ satisfies condition ($*$) for $(D, N \subset M),$ if 
 	\begin{quote}
 		$\mathbf{(*)}$ for  $(D, N \subset M):$ \, \, \, \, \, $d$ is a non-zero divisor on $N$ and $d \hgy{1}{I; M} = 0.$
 	\end{quote}
 	\end{defn}
	\indt When condition $(*)$ is satisfied we show, in section \ref{techlemmas}, that a weakened form of the Artin-Rees lemma applies uniformly across all sufficiently close perturbations (see lemma \ref{dLem}).  This is a key ingredient in the proof of lemma \ref{mainlemma}, which is a general technical tool for relating the $A$-module properties of the $\olper{M}.$\\
\indt	Note that when $I \subset R$ is generated by an $M$-regular sequence the condition is automatic: $\hgy{1}{I; M} = 0$ and therefore $1 \in A$  satisfies condition $(*)$ for $(D, N \subset M).$  In general, $(*)$ puts strong constraints on the ideal $I.$  For example, in certain circumstances it forces the height of $I$ to be as large as possible:
\begin{prop} \label{star}
 Suppose in the diagram above, $D,$ $d \hgy{1}{I; M} = 0$ for some non-zero $d \in A,$ where $M$ is a finite $R$-module. If there is a minimal prime $\mf{p} \subset R$ of $I$ with $\height{\mf{p}} = \height{I}$ such that $d \not\in \mf{p}$ and $M_\mf{p} \neq 0,$ then $\height{I} = \mu(I).$
\end{prop}
\begin{proof}
Choose minimal generators $\unideal{f} = (f_1, \dots, f_c) = I.$ Suppose that $\mf{p} \in \min(I)$ is a minimal prime with $\height{\mf{p}} = \height{I},$ and $d \not\in \mf{p}.$  Our assumption implies that $$\hgy{1}{\unideal{f}; M}_\mf{p} = \hgy{1}{\unideal{f/1}R_{\mf{p}}; M_{\mf{p}}} = 0.$$ \\
\indt $R_{\mf{p}}$ is local and $M_\mf{p} \neq 0$, so, by corollary 1.6.19 of \cite{bh1998}, the sequence $\frac{f_1}{1}, \dots, \frac{f_c}{1}$ is $M_{\mf{p}}$-regular. In particular, $\depth{\mf{p}R_{\mf{p}}}{M_{\mf{p}}} \ge c.$ \\
\indt Combining this with the inequalities, $ \dim R_{\mf{p}} = \height{I} \le c$ from the Krull height theorem, and $\depth{\mf{p}R_{\mf{p}}}{M_{\mf{p}}} \le \dim R_{\mf{p}},$ we get
	\begin{align*}
		\dim R_{\mf{p}} \le c \le \depth{\mf{p}R_{\mf{p}}}{M_{\mf{p}}} \le \dim R_{\mf{p}}
	\end{align*}
and so we have $\height{I} = \dim R_{\mf{p}} = c,$ as claimed.
\end{proof}
\begin{remark} \label{equidim}
	A particular case of this is worth highlighting.  Suppose that, in proposition \ref{star}, $R/I$ is equidimensional, $A$ is a domain, and $d\hgy{1}{I; M} = 0$ for some nonzero $d \in A$.  Since $A$ is a domain and the extension $A \xhookrightarrow{} \ol{R}$ is finite, $\mf{p} \cap A = 0$ for all $\mf{p} \in \minh{R/I}.$  Therefore $d \not\in \mf{p}$ for any minimal prime $\mf{p},$ of $\ol{R},$ and if $M_\mf{p} \neq 0$ for any minimal prime $\mf{p}$ of $\ol{R},$ the conclusion of the proposition, $\height{I} = \mu(I),$ holds.  
\end{remark} \\
	\indt The formulation of definition \ref{condition}, as well as the main results in section \ref{techlemmas}, are intended to be very general.  In practice, to take advantage of this theory one needs their data to fit into the very special setting of diagram $D$.  In this paper we achieve this through construction \ref{sopDiagram} (i.e. the Cohen structure theorem), so we are especially concerned with the case where $A$ is a regular local ring.  Because of the particular applications we have in mind, we are also most interested in the case when $N \subset M$ is a free $A$-submodule of $M$.  In these circumstances, any nonzero $d \in A$ such that $d \hgy{1}{I; M} = 0$ will satisfy $(*)$.  The next few statements are formulated with this case in mind. \\
\begin{lemma} \label{dimH1}	
	Suppose $(R, \mf{m}_R),$ $(A, \mf{m}_A),$ $\ol{R} = R/I,$ and the diagram $D$ are as above, and that $M$ is a finite $R$-module. \\	
Assume  $$\dim \dfrac{R}{\ann{R}{\hgy{1}{I; M}}} = \dim_{\ol{R}} \hgy{1}{I; M} < \dim \ol{R} \, .$$ 
Then, there is an element, $d \in A,$ such that $$\dim A/dA = \dim A - 1$$ and $d \hgy{1}{I; M} = 0.$	
\end{lemma}
\begin{proof}
	Let $K$ denote the kernel of the composition 
	\begin{align*}
	A \xhookrightarrow{} \ol{R} \twoheadrightarrow \ol{R}/\ann{\ol{R}}{\hgy{1}{I; M}} \, .
	\end{align*}
	The induced extension
	\begin{align*}
		A/K \xhookrightarrow{} \ol{R}/\ann{\ol{R}}{\hgy{1}{I; M}}
	\end{align*}
	continues to be module finite, and thus 
	\begin{align*}
		\dim A/K = \dim \ol{R}/\ann{\ol{R}}{\hgy{1}{I; M}} < \dim \ol{R} = \dim A \, .
	\end{align*}
	It follows that $K = A \cap \ann{\ol{R}}{\hgy{1}{I; M}}$ contains a parameter.
\end{proof}
\begin{remark} \label{dimH1_converseFilterReg_Reg}
	\begin{itemize}
	\item[(i)]	It is not hard to show that the converse of lemma \ref{dimH1} also holds, in the sense that if there is any $A$ fitting into diagram $D$ with a parameter, $d \in A,$ that annihilates $\hgy{1}{I; M}$ then $\dim_{\ol{R}} \hgy{1}{I; M} < \dim \ol{R}.$   \\
	\item[(ii)] The condition on the dimension of $\hgy{1}{I; M}$ in lemma \ref{dimH1} should be understood as an especially convenient weakening of the condition that $I$ be generated by an $M$-regular sequence. For example:
	\begin{itemize}
	\item[(ii.a)]	If $I \subset R$ happens to have a minimal generating set that forms an $M$-filter regular sequence, then $\ann{R}{\hgy{1}{I; M}}$ is $\mf{m}_R$-primary and hence $$\dim_{\ol{R}} \hgy{1}{I; M} = 0$$ Therefore the conditions of lemma \ref{dimH1} are automatically satisfied if $\ol{R}$ has positive dimension (see 2.1 of \cite{polstra2019nilpotence}, section 1.2 of \cite{maddox2019sufficient}, or \cite{Mafilter} for details about filter-regular sequences).   
	\item[(ii.b)]  When $R$ is equicharacteristic and complete then one can apply construction \ref{sopDiagram} (see remark \ref{mixedchar} for the mixed characteristic case) to get a regular local domain $A$ fitting together with $R$ and $R/I$ into a commutative diagram of the desired form, $D.$  Take $M =R$, for simplicity, and assume $\dim_{\ol{R}} \hgy{1}{I; R} < \dim \ol{R}$.  Then, as discussed in remark \ref{equidim}, if $I \subset R$ is equidimensional lemma \ref{star} implies that $I$ must be a parameter ideal.  In this case, the conclusion of proposition \ref{dimH1altsupp} with $M = R$ implies that $R_{\mf{q}}$ is Cohen-Macualay for each $\mf{q} \in \min\left(R/I\right).$  If, in addition, $R/I$ has no embedded primes, then $I$ is actually generated by a regular sequence, by corollary 2 of \cite{eisenbud1977remarks}. \\
	\end{itemize}
	\item[(iii)] We should also note that assumptions of lemma \ref{dimH1} are trivially satisfied when $M_\mf{q} = 0$ for every $\mf{q} \in \minh{R/I}.$ 	
\end{itemize}
\end{remark} 
\vspace{.1in}
\indt  When $M$ has support at a prime in $\minh{R/I},$ lemma \ref{dimH1} has an equivalent formulation. 
\begin{prop} \label{dimH1altsupp}
Suppose $(R, \mf{m}_R)$ is a Noetherian local ring, $I \subset R$ is an ideal, and $M$ is a finite $R$-module such that $\minh{R/I} \cap \supp{R}{M} \neq 0$.  Then the following are equivalent:
\begin{itemize}
\item[(i)] $\dim R /\ann{R}{ \hgy{1}{I; M}} <  \dim R/I$;
\item[(ii)] for any minimal generators $(f_1, \dots, f_c) = I,$  $f_1/1, \dots, f_c/1$ is an $M_\mf{q}$-regular sequence for every $\mf{q} \in \minh{R/I} \cap \supp{R}{M}$;
\item[(iii)] there are minimal generators $(f_1, \dots, f_c) = I,$ such that $f_1/1, \dots, f_c/1$ is an $M_\mf{q}$-regular sequence for every $\mf{q} \in \minh{R/I} \cap \supp{R}{M}$.
\end{itemize}
\end{prop}
\begin{proof}
((i) $\Longrightarrow$ (ii)) \\
\indt The assumption $\dim R /\ann{R}{ \hgy{1}{I; M}} <  \dim R/I,$ means that $$\ann{R}{\hgy{1}{I; M}} \not \subset \mf{q}$$ for any $\mf{q} \in \minh{R/I}.$  Hence, $$\left(\hgy{1}{I; M}\right)_{\mf{q}} = 0,$$ for all $\mf{q} \in \minh{R/I}.$  \\
\indt Therefore, for any minimal generators, $(f_1, \dots, f_c) = I,$ and every $\mf{q} \in \minh{R/I},$
\begin{align*}
	\hgy{1}{(f_1/1, \dots, f_c/1) ; M_\mf{q}} \cong \left(\hgy{1}{I; M}\right)_{\mf{q}} = 0 \, .
\end{align*}
Thus, if $M_\mf{q} \neq 0,$ then $f_1/1, \dots, f_c/1$ is $M_\mf{q}$-regular by corollary 1.6.19 of \cite{bh1998}. \\
\vspace{.1in}
(ii) obviously implies (iii) \\
\vspace{.1in}
((iii) $\Longrightarrow$ (i)) \\
\indt Suppose $(f_1, \dots, f_c) = I$ are minimal generators, which form an $M_\mf{q}$-regular sequence in $R_\mf{q},$ for every $\mf{q} \in \minh{R/I} \cap \supp{R}{M}.$  First note that 
\begin{align*}
	\left(\hgy{1}{I; M}\right)_{\mf{q}} \cong \hgy{1}{(f_1/1, \dots, f_c/1) ; M_\mf{q}} = 0
\end{align*}
for every $\mf{q} \in \minh{R/I} \setminus \supp{R}{M},$ since, $M_{\mf{q}} = 0$ for these primes.  And, by assumption we also have $$\hgy{1}{(f_1/1, \dots, f_c/1) ; M_\mf{q}} = 0$$ for every $\mf{q} \in \minh{R/I} \cap \supp{R}{M}$ --- again by corollary 1.6.19 of \cite{bh1998}.\\
\indt  It follows that $\ann{R}{\hgy{1}{I; M}} \not \subset \mf{q}$ for any $\mf{q} \in \minh{R/I} \cap \supp{R}{M}.$  Noting that $I \subset \ann{R}{\hgy{1}{I; M}}$ we conclude that
\begin{align*}
	\dim R /\ann{R}{ \hgy{1}{I; M}} <  \dim R/I \, .
\end{align*}
\end{proof}
\indt A quick word of caution: this condition on the dimension of $\mathbf{H}_1$ is somewhat subtle.  The common source of confusing is that $\hgy{1}{I; M}$ is computed using minimal generators for $I$ in $R$;  after localizing at a prime, $\mf{q} \in \spec{R},$ these elements may no longer be minimal generators for $IR_{\mf{q}}$ in $R_{\mf{q}},$ and in this case $\hgy{1}{I; M}_{\mf{q}} \neq \hgy{1}{IR_{\mf{q}}; M_{\mf{q}}}.$  For example, statements (ii) and (iii) of proposition \ref{dimH1altsupp} may be interpreted as saying that the minimal generators of $I$ 'generically form a regular sequence on $M$'.  This is not the same thing as $I$ being 'generically generated by an $M$-regular sequence'.  Indeed, every prime $\mf{p} \subset S$ in a regular local ring, $S$, is generically generated by a regular sequence, since $S_\mf{p}$ is regular --- however, minimal generators of $\mf{p}$ in $S$ will only form a regular sequence on $S_{\mf{p}}$ when $\mf{p}$ is generated by part of a system of parameters (and, in this case, must already be a regular sequence on $S$).  
\section{Technical Tools} \label{TechTools}
\subsection{Lemmas} \label{techlemmas}
\indt The first result of this section is needed to show the Artin-Reese number in lemma \ref{dLem} is well defined.  
\begin{lemma} \label{isoComp}
	Suppose that
	$$
	\begin{tikzcd}
		B \arrow[r, "\beta"] \arrow[d, "\alpha"]
			& C \arrow[d, "\alpha'"] \\
		B \arrow[r, "\beta'"]
			& C
	\end{tikzcd}
$$ is a commutative diagram of $R$-modules, and that $\alpha: B \to B$ and $\alpha':C \to C$ are isomorphisms. Then
		\begin{itemize}
			\item[(i)] $\knl{\beta'} = \alpha(\knl{\beta}),$ and $\knl{\beta} = \alpha^{-1}(\knl{\beta'})$;
			\item[(ii)] For any ideal $L \subset R,$ 
			\begin{align*} 
				\AR{L}{\knl{\beta} \subset B}{R} = \AR{L}{\knl{\beta'} \subset B}{R} \, .
			\end{align*}
		\end{itemize}
	\end{lemma}
	\begin{proof}
		\begin{itemize}
			\item[(i)] Clearly 
				\begin{align*}
					x \in \knl{\beta} \Longleftrightarrow \alpha' \circ \beta (x) = 0 \Longleftrightarrow \beta' \circ \alpha(x) = 0 \Longleftrightarrow \alpha(x) \in \knl{\beta'}
				\end{align*}
			\item[(ii)] $\alpha$ is an isomorphism, so for any submodules $K, K' \subset B,$ 
				\begin{align*}
					\alpha( K \cap K' ) \subset \alpha(K) \cap \alpha(K') = \alpha \circ \alpha^{-1} \left(\alpha(K) \cap \alpha(K')\right) \subset \alpha\left(\alpha^{-1}\alpha(K) \cap \alpha^{-1}\alpha(K') \right) = \alpha(K \cap K').
				\end{align*}			
				Notice, in particular, that $$\alpha(K \cap K') = \alpha(K) \cap \alpha(K').$$  Of course, $\alpha$ is $R$-linear, so $\alpha(L^kK) = L^k \alpha(K),$ for any $k \ge 0,$ and any $R$-submodule $K \subset B.$ \\
				So, for any $n+t \ge 0,$ we get
				\begin{align*}
					\alpha\left( L^{n+t} B \cap \knl{\beta}\right) = L^{n+t}B \cap \alpha(\knl{\beta}) = L^{n+t} B \cap \knl{\beta'}.
				\end{align*}
				Therefore,
				\begin{align*}
					 L^{n+t} B \cap \knl{\beta} = L^n \left(L^tB \cap \knl{\beta}\right) \Longleftrightarrow L^{n+t}B \cap \knl{\beta'} &= \alpha \left(L^n \left(L^tB \cap \knl{\beta}\right) \right)\\
					 &= L^n \left(L^tB \cap \knl{\beta'}\right)\, .
				\end{align*}
		\end{itemize}
	\end{proof} 
\indt Next, for lack of a precise reference, we establish some well known length bounds for use later on. \\
\begin{lemma} \label{filtration}
Suppose $(S, \mf{m}_S)$ is a local ring, $J \subset S$ is an ideal, and $M$ is an $S$-module.\\
\begin{itemize}
\item[(i)]	Given any elements, $y ,z \in S$, there is an exact sequence
	\begin{align*}
		0 \to \dfrac{\col{(J, zy)M}{z}{M}}{(J, y)M} \to \dfrac{M}{(J, y)M} \to \dfrac{(J, z)M}{(J, yz)M} \to 0
	\end{align*} 
\item[(ii)] If $M$ is finitely generated, $z_1, \dots, z_k \in \mf{m}_S$ are such that $(J, z_1, \dots, z_k)$ is $\mf{m}_S$-primary, and $n_1, \dots, n_k \ge 1,$ then
	\begin{align*}
		\len{S}{M/(J, z_1^{n_1}, \dots, z_k^{n_k})M} \le n_1 n_2 \dots n_k \len{S}{M/(J, z_1, \dots, z_k)M} \, .
	\end{align*}
\end{itemize}
\end{lemma}
\begin{proof}
\begin{itemize}
\item[(i)]	The composition,
	\begin{align*}
		M \to \dfrac{M}{(J, yz)M} \xrightarrow{\cdot z} \dfrac{(J, yz, z)M}{(J, yz)M} = \dfrac{(J, z)M}{(J, yz)M}
	\end{align*}
	is surjective, and with kernel $\col{(J, zy)M}{z}{M}.$  There is a containment of submodules $(J, y)M \subset \col{(J, zy)M}{z}{M},$ and thus an induced surjection,
	\begin{align*}
		\dfrac{M}{(J, y)M} \to \dfrac{(J, z)M}{(J, zy)M},
	\end{align*}
	with kernel $\dfrac{\col{(J, zy)M}{z}{M}}{(J, y)M}.$
\item[(ii)]  When $k=1$, we have,
	\begin{align*}
	\len{S}{M/(J, z^n)M} = \len{S}{M/(J, z)M} + \sum_{k=1}^{n-1} \len{S}{\dfrac{(J, z^{k})M}{(J, z^{k+1})M}}
	\end{align*}
	By (i) above, for each of these $k$,
	\begin{align*}
		\len{S}{\dfrac{(J, z^{k})M}{(J, z^{k+1})M}} = \len{S}{M/(J, z)M} - \len{S}{\dfrac{\col{(J, z^{k+1})M}{z^k}{M}}{(J, z)M}} \, .
	\end{align*} \\ 
	So,
	\begin{align*}
		\len{S}{M/(J, z^n)M} &= \len{S}{M/(J, z)M} + \sum_{k=1}^{n-1} \left( \len{S}{M/(J, z)M} - \len{S}{\dfrac{\col{(J, z^{k+1})M}{z^k}{M}}{(J, z)M}}\right) \\
		&= n\len{S}{M/(J, z)M} - \sum_{k=1}^{n-1} \len{S}{\dfrac{\col{(J, z^{k+1})M}{z^k}{M}}{(J, z)M}} \, .
	\end{align*}
The general case follows by induction.
\end{itemize}
\end{proof}
\indt With those preliminary results out of the way, we turn to the main results of this section.  Lemma \ref{dLem} shows that, up to multiplication by an element that kills the relevant first Koszul homology module, a kind of 'coefficient-wise' Artin-Rees lemma applies uniformly when the minimal generators of an ideal are replaced by a perturbation.  This result, which is interesting in it's own right, is a key part of the proof of lemma \ref{mainlemma}.  
\begin{lemma}\label{dLem}
	Suppose $M$ is a finitely generated module over a Noetherian local ring $(R, \mf{m}_R),$ and that $I, J, L \subset R$ are ideals, and let $c = \mu (I)$ be the minimal number of generators of $I$.  Suppose that $d \in R$ annihilates the first Koszul homology of $I$ with coefficients in $M,$
	\begin{align*}
		d \hgy{1}{I; M} = 0
	\end{align*}
	and let $t = \AR{L}{ \ker\left(\partial^{M}_{1, I}\right) \subset M^{\oplus c}}{R}.$ \\
	\indt Suppose $m_1, \dots, m_h \in M,$ that $k \ge t,$ and $$z_1m_1 + \dots + z_hm_h \in L^{k+q}IM,$$ where the coefficients $z_1, \dots, z_h \in L^kI \subset R.$ \\
	Choose any minimal generators $\unideal{f} = (f_1, \dots, f_c) = I$ and write $$z_i = y_{i1}f_1 + \dots + y_{ic} f_c$$ with each $y_{ij} \in L^k.$ \\
	 Then, for any $\epsilon_1, \dots, \epsilon_c \in J,$ 
	 \begin{align*}
	 	d \left(\wt{z}_1m_1+\dots +\wt{z}_h m_h \right) \in \left(L^{k-t}J + dL^{k+q}\right)\unideal{f+\epsilon}M 	
	 \end{align*}
	 where
	 \begin{align*}
	 	\wt{z}_i = y_{i1} (f_1 + \epsilon_1) + \dots + y_{ic} (f_c+\epsilon_c) \, .
	 \end{align*}
	\end{lemma}
	\begin{proof}
	First note that the Artin-Rees number $t = \AR{L}{ \ker\left(\partial^{M}_{1, I}\right) \subset M^{\oplus c}}{R}$ is well defined, by lemma \ref{isoComp}, since the $\mathbf{K}\left(\unideal{f} ; M\right)$ are {\it isomorphic complexes} for all choices of minimal generators $\unideal{f} = I.$ \\
	\vspace{.2in}
	\indt We have $m = z_1m_1 + \dots + z_hm_h \in L^{k+q}IM,$
	so there are $n_1, \dots, n_\ell \in M,$ and $s_{ij} \in L^{k+q},$ $i=1, \dots, \ell$ and $j = 1, \dots, c,$ such that
	\begin{align*}
	 m &=  \sum_{i=1}^h \left(\sum_{j=1}^c y_{ij} f_j\right)m_i = \sum_{i=1}^\ell \left(\sum_{j=1}^c s_{ij} f_j\right)n_i.
	\end{align*}	
	 Setting $$\alpha_j = \left(\sum_{i=1}^h y_{ij} m_i\right) - \left(\sum_{i=1}^\ell s_{ij}n_i \right),$$ for each $j = 1, 2, \dots, c,$ the equality above can be rearranged into the form 
	 \begin{align*}
	 f_1 \alpha_1 + f_2 \alpha_2 + \dots + f_c \alpha_c = 0 \, .
	 \end{align*} \\
	 Noting that all the coefficients appearing in the $\alpha$'s are in $L^k,$ and that $k \ge \AR{L}{ \ker\left(\partial^{M}_{1, I}\right) \subset M^{\oplus c}}{R}$ we have
	 \begin{align*}
	 \left(\alpha_1, \dots, \alpha_c\right) \in \ker \left(\partial^{M}_{1, \unideal{f}}\right) \cap L^k M^{\oplus c} \subset L^{k-t}\ker \left(\partial^{M}_{1, \unideal{f}}\right) \, .
	 \end{align*}  
	 By assumption, the element $d$ annihilates $\hgy{1}{\unideal{f}; M},$ and hence we have, 
	 \begin{align} \label{one}
	 	d \left(\alpha_1, \dots, \alpha_c\right) \in L^{k-t} \im{\partial^{M}_{2, \unideal{f}}}
	 \end{align}
	 Now, choose any $\epsilon_1, \dots, \epsilon_c \in J,$ and note that the difference 
	 \begin{align*}
	 	\partial^{M}_{2, \unideal{f}} - \partial^{M}_{2, \unideal{f+\epsilon}}: M^{\oplus \binom{c}{2}} \to M^{\oplus c}
	 \end{align*}
	 takes values in $J M^{\oplus c},$ and there is an inclusion 
	 \begin{align} \label{two}
	 	\im{\partial^{M}_{2, \unideal{f}}} \subset \im{\partial^{M}_{2, \unideal{f+\epsilon}}} + J M^{\oplus c} \, .
	 \end{align}
	 Combining \ref{one} and \ref{two} we have shown,
	 \begin{align*}
	 	d \left(\alpha_1, \dots, \alpha_c\right) \in L^{k-t} \im{\partial^{M}_{2, \unideal{f}}} &\subset L^{k-t} \left(\im{\partial^{M}_{2, \unideal{f+\epsilon}}} + J M^{\oplus c} \right) \\
	 	&\subset \ker\left(\partial^{M}_{1, \unideal{f+\epsilon}}\right) + L^{k-t}JM^{\oplus c} \, .
	 \end{align*}
	 So, applying $\partial_{1, \unideal{f+\epsilon}}^{M}$ will give
	\begin{align*}
		\partial_{1, \unideal{f+\epsilon}}^{M} \left(d (\alpha_1, \dots, \alpha_c) \right) \in L^{k-t}J \,\, \im{\partial_{1, \unideal{f+\epsilon}}^{M}} = L^{k-t}J\unideal{f+\epsilon} M \, .
	\end{align*}
	Unwinding this, we have shown that
	\begin{align*}
	\partial_{1, \unideal{f+\epsilon}}^{M} \left(d (\alpha_1, \dots, \alpha_c) \right) &= d(f_1+\epsilon_1)\alpha_1 + \dots + d(f_c+\epsilon_c)\alpha_c \\
	&= d(f_1+\epsilon_1) \left(\left(\sum_{i=1}^h y_{i1} m_i\right) - \left(\sum_{i=1}^\ell s_{i1}n_i \right)\right) + \dots \\ &\hspace{1in} \dots + d(f_c+\epsilon_c) \left(\left(\sum_{i=1}^h y_{ic} m_i\right) - \left(\sum_{i=1}^\ell s_{ic}n_i \right)\right) \\
	&= d \left(\wt{z}_1 m_1+\dots +\wt{z}_h m_h \right) - d\beta_1 n_1 - \dots - d\beta_{\ell} n_\ell \\
	&\in L^{k-t}J\unideal{f+\epsilon} M
	\end{align*}
	with each $\beta_i = s_{i1}(f_1+\epsilon_1) + \dots + s_{ic}(f_c+\epsilon_c) \in L^{k+q}\unideal{f+\epsilon}R.$ \\
	\end{proof} 
\indt	Lemma \ref{dLem} was partly inspired by \cite{Mafilter}, corollary 3.6, where the authors show there is a similar kind of uniform bound on Artin-Rees numbers applying across small perturbations of a filter regular sequence.  Now we are ready to prove the main lemma of this section.
\begin{lemma} \label{mainlemma}
	Suppose $(R, \mf{m}_R)$ and $(A, \mf{m}_A)$ are Noetherian local rings, that $I \subset R$ is an ideal, and $D$ is a commuting diagram of local rings and local ring maps as before.\\
Suppose that $M$ is a finitely generated $R$-module, and $N \subset M$ is a finitely generated $A$-submodule.
 	Assume:
 	\begin{itemize}
 	\item[(i)] There is a $d \in A$ satisfying condition ($*$) for $(D, N \subset M)$; \\
 	\item[(ii)] $\depth{\mf{m}_A}{N} \ge 1,$ so there is an $x \in \mf{m}_A$ which is a nzd on $N$; \\
 	\item[(iii)] $N \cap IM = 0.$
 	\end{itemize}
 	Then, there is a $T = T\left(d, D, N, M\right) \in \mathbb{N},$ such that, for all minimal generators $(f_1, \dots, f_c) = \unideal{f} = I,$ and any $\epsilon_1, \dots, \epsilon_c \in \mf{m}_R^T,$ $$N \cap \unideal{f+\epsilon}M = 0.$$
\end{lemma}
	\begin{proof}
	Assume (i), (ii) and (iii) hold. \\
	\indt  By assumption $N \cap I M = 0,$ and with this in mind, we will freely identify the finite $A$ module $N$ with it's image in $\ol{M} = M/IM.$ \\
	\indt Label the following Artin-Rees numbers: 
\begin{align*}
	t_1 &= \AR{\mf{m}_A R}{IM \subset M}{R}, \\
	t_2 &= \AR{\mf{m}_A R}{ \knl{\partial^{M}_{1, I}} \subset M^{\oplus c}}{R}, \\
	t_3 &= \AR{\mf{m}_A}{dN \subset \ol{M}}{A}, \\
	t_4 &= \AR{\mf{m}_A}{N \subset \ol{M}}{A}. \\
\end{align*}
\indt Fix $t \ge \max(t_1, t_2, t_3, t_4),$ set $H = 3t+1,$ and choose $T$ large enough that $$\mf{m}_R^T \subset \mf{m}_R \left(I + \mf{m}_A^H R\right).$$	\\
\indt Suppose $(f_1, \dots, f_c) = I$ are minimal generators and $\epsilon_1, \dots, \epsilon_c \in \mf{m}_R^T.$\\
\indt We are interested in the ideal $(f_1+\epsilon_1, \dots, f_c+\epsilon_c),$ and, according to lemma \ref{mprimaryepsilons} (ii), for the sake of what follows, we may assume $\epsilon_1, \dots, \epsilon_c \in \mf{m}_A^H R$ --- after harmlessly exchanging the $f$'s with a different set of minimal generators, if necessary.  \\
\begin{claim} \label{clm1}
	Suppose that $m \in N,$ and $d^k m \in \mf{m}_A^n M,$ with $n > kt.$  Then $$m \in \mf{m}_A^{n - kt} N \, .$$
\end{claim}
\begin{proof}
First suppose that $m \in N,$ and $dm \in \mf{m}_A^n M$ with $n \ge t.$  Then, going modulo $IM,$ and identifying $N$ with it's image, 
\begin{align*}
	d m \in \left(dN\right) \cap \left(\mf{m}_A^n \ol{M} \right) \subset \mf{m}_A^{n-t} dN.
\end{align*}
By assumption, $d$ is a nzd on $N$, and so $m \in \mf{m}_A^{n-t}N.$  This establishes the claim for $k=1,$ and the larger values follow by induction.
\end{proof}
\begin{claim} \label{clm2}
	For $\ell \ge t,$  
	\begin{align*}
	d \left[ \left( \mf{m}_A^{\ell + H - t} N \right) \cap \left(\mf{m}_A^\ell \unideal{f+\epsilon} M\right)\right] \subset \left[\left(\mf{m}_A^{\ell + 2H - 3t}N\right) \cap \left(\mf{m}_A^{\ell + H - 2t}\unideal{f+\epsilon}M\right)\right] \, . \\
	\end{align*}
\end{claim}
\begin{proof}
Suppose $$m \in \left( \mf{m}_A^{\ell + H - t} N \right) \cap \left(\mf{m}_A^\ell \unideal{f+\epsilon} M\right),$$ so that we may write
\begin{align*}
	m &= \left(\sum_{j=1}^c y_{1j}(f_j + \epsilon_j)\right) m_1 + \dots + \left(\sum_{j=1}^c y_{hj}(f_j + \epsilon_j)\right) m_h \in \mf{m}_A^{\ell + H -t} N \subset \mf{m}_A^{\ell + H -t} M,
\end{align*}
with $m_1, \dots, m_h \in M$ and the $y_{ij} \in \mf{m}_A^{\ell}R.$ \\
\vspace{.1in}
The products $y_{ij} \epsilon_j$ are all in $\mf{m}_A^{\ell + H}R,$ and thus the difference
\begin{align} \label{sum}
	m - &\left(\sum_{j=1}^c y_{1j}\epsilon_j\right)m_1 - \dots - \left(\sum_{j=1}^c y_{hj}\epsilon_j\right)m_h = \left(\sum_{j=1}^c y_{1j} f_j\right)m_1 + \dots + \left(\sum_{j=1}^c y_{hj} f_j\right)m_h
\end{align}
belongs to $\left[\left(\mf{m}_A^{\ell+H-t} M\right) \cap I M \right] \subset \mf{m}_A^{\ell+H-2t} IM.$ \\
The right hand side of (\ref{sum}) is in $\mf{m}_A^{\ell+H-2t} IM$ and each of the $y_{ij} \in \mf{m}_A^{\ell},$ so we may apply lemma \ref{dLem}, with $L = \mf{m}_A R$ and  $J = \mf{m}_A^HR,$ to conclude that
\begin{align*}
	dm &= d\left[\left(\sum_{j=1}^c y_{1j} (f_j+ \epsilon_j)\right)m_1 + \dots + \left(\sum_{j=1}^c y_{hj} (f_j+\epsilon_j)\right)m_h \right] \\ 
	&\in \left(\mf{m}_A^{\ell+H-t}R + d \mf{m}_A^{\ell+H-2t}R\right)\unideal{f+\epsilon}M \subset \mf{m}_A^{\ell+H-2t} \unideal{f+\epsilon}M \, .
\end{align*} \\
Now, $dm \in N$ as well.  Identifying $N$ with it's image in $\ol{M},$ and noting that $$\unideal{f+\epsilon} \dfrac{R}{I} \subset \mf{m}_A^H \dfrac{R}{I},$$ we see that $$dm \in N \cap \left(\mf{m}_A^{\ell+2H-2t} \ol{M}\right) \subset \mf{m}_A^{\ell+2H-3t}N,$$ and the claim is proved.
\end{proof}
Now to establish the lemma. \\
\indt Applying claim \ref{clm2} $k$ times and using $H = 3t+1,$ we have
	\begin{align*}
	d^k \left[ \left( \mf{m}_A^{\ell + H - t} N \right) \cap \left(\mf{m}_A^\ell \unideal{f+\epsilon} M\right)\right] &\subset \left[\left(\mf{m}_A^{\ell + (k+1)H - (2k+1)t}N\right) \cap \left(\mf{m}_A^{\ell + k(H - 2t)}\unideal{f+\epsilon}M\right)\right] \\
	&\subset \mf{m}_A^{\ell + (k+1)(3t+1) - (2k+1)t} M\\
	&\subset \mf{m}_A^{\ell + kt + 2t + k + 1} M \subset \mf{m}_A^{\ell + kt + k + 1} M \, . \\
	\end{align*}
Now claim \ref{clm1} gives,
	\begin{align*}
	\left[ \left( \mf{m}_A^{\ell + H - t} N \right) \cap \left(\mf{m}_A^\ell \unideal{f+\epsilon} M\right)\right] \subset \mf{m}_A^{\ell + k + 1} M  \, .
	\end{align*}
	This inclusion is true for every $k,$ so we conclude that $$\left( \mf{m}_A^{\ell + H - t} N \right) \cap \left(\mf{m}_A^\ell \unideal{f+\epsilon} M\right) = 0$$ for any $\ell \ge t,$ by Krull intersection (note that $\mf{m}_A^{\ell+k+1}M = \left(\mf{m}_AR\right)^{\ell+k+1}M$). \\
	\vspace{.1in}
	So, take any $\ell \ge t$ and consider 
	\begin{align*}
	B = \mf{m}_A^\ell \left[N \cap \unideal{f+\epsilon}M \right] \subset N \cap \mf{m}_A^\ell \unideal{f+\epsilon}M.
	\end{align*}
	Going modulo $I M,$ and identifying $B$ and $N$ with their images, we have
	\begin{align*}
	B \subset \left[N \cap \mf{m}_A^\ell \unideal{f+\epsilon} \ol{M} \right] \subset \left[N \cap \mf{m}_A^{\ell + H} \ol{M}\right] \subset \mf{m}_A^{\ell + H - t} N \, .
	\end{align*}
	It follows that $$B \subset \left( \mf{m}_A^{\ell + H - t} N \right) \cap \left(\mf{m}_A^\ell \unideal{f+\epsilon} M\right) = 0.$$ \\
	So, we have shown that for every $\ell \ge t,$ $$\mf{m}_A^\ell \left[N \cap \unideal{f+\epsilon}M \right] = 0$$
	We are assuming that $\mf{m}_A$ contains a nzd on $N,$ and therefore every element of $N \cap \unideal{f+\epsilon}M$ must be zero.
	\end{proof}
	\begin{obs}  More generally, if we don't assume that $N$ has positive depth on $A$, the proof shows that if the $\epsilon$'s are chosen as prescribed, all of the $N \cap \unideal{f+\epsilon}M$ are uniformly annihilated by $\mf{m}_A^t.$
	\end{obs} \\
\subsection{Refinements} \label{refine}
\indt We make repeated use of a particular application of lemma \ref{mainlemma}. 
\begin{thm} \label{freeper}
Suppose that $(R, \mf{m}_R)$ and $(A, \mf{m}_A)$ are Noetherian local rings and that $I \subset R$ is an ideal.  Assume these are related as in diagram $D$.  Set $\ol{R} = R/I$ and assume that $A$ satisfies $\depth{\mf{m}_A}{A} \ge 1$.  \\
\indt Suppose $M$ is a finite $R$-module such that $d \hgy{1}{I; M} = 0$ for some $d \in A$ which is a nzd in $A$. Let $m_1, \dots, m_n \in M$ be elements such that the image of 
\begin{align*}
	N := Am_1 + \dots + Am_n
\end{align*}
under the quotient map $M \to \ol{M} = M/IM$ is a free $A$-module of rank $n$.  \\
\indt Then, there is a $T = T(D, M, N),$ such that for all $\epsilon_1, \dots, \epsilon_c \in \mf{m}_R^T$ and all minimal generators $(f_1, \dots, f_c) = I$ the image of $N$ in $\olper{M}$ is a free $A$-module of rank $n$.
\end{thm}
\begin{proof}
  We show the conditions of lemma \ref{mainlemma} are satisfied: \\ 
\indt First of all, any nonzero element of $N \cap IM$ would be mapped to a non-trivial relation in the image of $N$ under the quotient map, $M \to \ol{M}.$  Since we are assuming this image is free of rank $n$, we must have $N \cap IM = 0,$ and, consequently, $N \cong A^{\oplus n}.$  It follows that the nzd $d \in A$ is a nzd on $N$, and thus $d$ satisfies condition $(*)$ for $(D, N \subset M).$  Finally, since $N$ is free, $\depth{\mf{m}_A}{N} = \depth{\mf{m}_A}{A} \ge 1$.  \\
\indt By lemma \ref{mainlemma} there is a $T$ such that $N \cap \unideal{f+\epsilon}M = 0$ for all minimal $\unideal{f} = I$ and all $\epsilon_1, \dots, \epsilon_{\mu(I)} \in \mf{m}_R^T.$  For these $\unideal{f+\epsilon},$ $N \cong A^{\oplus n}$ is isomorphic to it's image in $M/\unideal{f+\epsilon}M.$
\end{proof}
\begin{remark}
 As noted in remark \ref{injmap}, applying theorem \ref{freeper} to the free $A$-module $A \cdot 1 \subset R$ shows that the compositions, $A \xhookrightarrow{} R \twoheadrightarrow \olper{R},$ are injections for large $T$.  So, when $R$ and $A$ are complete, lemma \ref{mgenslemma} guarantees these are module finite extensions, $A \xhookrightarrow{} \olper{R},$ and for $T \gg 0$.  In particular, this means that the dimensions of sufficiently small perturbations of $\ol{R}$ will all agree. 
\end{remark} \\
\vspace{.1in}
\indt Now we have arrived at the promised extension of corollary \ref{rankle}.
\begin{cor} \label{rankper}
	Suppose that $(R, \mf{m}_R), (A, \mf{m}_A)$ and $(\ol{R}, \mf{m}_{\ol{R}})$ are complete local rings related as in the commuting diagram, $D$, with $\ol{R} = R/I$ and $A$ a domain.  \\
\indt Assume that $M$ is a finite $R$-module such that $d \hgy{1}{I; M} = 0$ for some $d \neq 0$ in $A$. \\
\indt Then, there is a $T = T(D, M),$ such that for all $\epsilon_1, \dots, \epsilon_c \in \mf{m}_R^T$ and all minimal generators $(f_1, \dots, f_c) = I$
	\begin{align*}
		\rank{A}{\ol{M}} = \rank{A}{\olper{M}} \, .
	\end{align*}
\end{cor}
\begin{proof}
If $A$ is a field, then $\ol{R}$ has Krull dimension $0$ and $I$ must be $\mf{m}_R$-primary.  In this case, corollary \ref{nakcordim0} shows that, in fact, all of the $\olper{M}$ are equal to $\ol{M}$ if the $\epsilon$'s are in a sufficiently high power of $\mf{m}_R$. \\
\indt If $\ol{M}$ is torsion as an $A$-module, this is already settled in corollary \ref{rankle}. \\
\indt It remains to establish the claim when $A$ is a local domain of dimension $\ge 1$ and $\ol{M}$ has positive rank as an $A$-module. \\ 
\indt Set $k = \rank{A}{\ol{M}},$ and choose elements, $n_1, \dots, n_k \in M$ that map to a basis for a free submodule of $\ol{M}$ of maximal $A$-rank, $k,$ and set $$N = An_1 + \dots + An_k.$$ By theorem \ref{freeper} there is a $T = T(d, D, N, M),$ such that for every $\epsilon_1, \dots, \epsilon_c \in \mf{m}_R^T$ and every minimal $(f_1, \dots, f_c) = I,$ the image of $N$ in $\olper{M}$ is free of rank $k$ over $A$. \\
\indt It follows that 
\begin{align*}
	\rank{A}{\olper{M}} \ge \rank{A}{\ol{M}}.
\end{align*} 
On the other hand, by corollary \ref{rankle}, if $T$ is large enough, the opposite inequality also holds,
\begin{align*}
	\rank{A}{\olper{M}} \le \rank{A}{\ol{M}}.
\end{align*} 
Thus, for $T \gg 0$, we get equality 
\begin{align*}
	\rank{A}{\olper{M}} = \rank{A}{\ol{M}}
\end{align*} 
and the result is proved.
\end{proof}
\indt Before moving on to discuss multiplicities, it is worth mentioning another interesting application of the results above.  Assume the setup of theorem \ref{freeper}, and that the conditions of the theorem are satisfied with $M = R.$  Also assume that $R$ and $A$ are complete, so that lemma \ref{mgenslemma} applies.  For the sake of exposition, let $(f_1, \dots, f_c) = I$ denote any minimal generators for $I,$ and suppose $\epsilon_1, \dots, \epsilon_c \in \mf{m}_R ^T,$ for some $T,$ to be specified below. \\
\vspace{.1in}
\indt If the finite map in our diagram, $A \xhookrightarrow{} \ol{R},$ is flat, $\ol{R}$ is isomorphic to a finite free $A$-module. Assume this is the case.  Then elements, $r_1, \dots, r_n \in R,$ that map to minimal $A$-generators for $\ol{R}$ span a free $A$-submodule of $\ol{R},$ namely $\ol{R}$ itself.  According to lemma \ref{mgenslemma}, if $T$ is large enough, these elements will continue to be minimal $A$-generators for $R/\unideal{f+\epsilon},$ and, by theorem \ref{freeper}, for $T \gg 0$, they will continue to span a free $A$-module.  This means that the induced maps, $A \xhookrightarrow{} R/\unideal{f+\epsilon},$ are also flat for $T \gg 0.$  Furthermore, these values of $T$ have been constructed, in the proofs of lemmas \ref{mgenslemma} and \ref{mainlemma}, such that $I + \mf{m}_AR = \unideal{f+\epsilon} + \mf{m}_A R.$  In other words, the closed fibers over $A$,
\begin{align*}
	\olper{R}/\mf{m}_A\olper{R} = \dfrac{R}{\unideal{f+\epsilon} + \mf{m}_A R} = \dfrac{R}{I + \mf{m}_AR} = \ol{R}/\mf{m}_A\ol{R},
\end{align*}     
are all equal.  We record this as theorem \ref{flatfiberper}. 
\begin{thm} \label{flatfiberper}
  Suppose $(R, \mf{m}_R), (A, \mf{m}_A)$ are complete Noetherian local rings, that $I \subset R$ is an ideal, and that $R$, $A$ and $\ol{R} = R/I$ are related as in diagram $D$.  Assume $\depth{\mf{m}_A}{A} \ge 1,$ and that $d \hgy{1}{I; R} = 0,$ for some nzd $d \in A.$ \\
\indt  Suppose the module finite map, $A \xhookrightarrow{} \ol{R},$ appearing in $D$ is flat.  Then, there is a $T \in \mathbb{N}$ such that for all $\epsilon_1, \dots, \epsilon_c \in \mf{m}_R^T$ and all minimal generators $(f_1, \dots, f_c) = I,$
\begin{itemize}
	\item[(i)] the induced map, $A \xhookrightarrow{} \olper{R},$ is flat;
	\item[(ii)] $\dfrac{\olper{R}}{\mf{m}_A \olper{R}} \cong \dfrac{\ol{R}}{\mf{m}_A \ol{R}}$.
\end{itemize}
\end{thm} 
\section{Applications to Multiplicity} \label{AppsMult}
\indt  Suppose that $(S, \mf{m}_S)$ is a Noetherian local ring of Krull dimension $a$, $J \subset S$ is an $\mf{m}_S$-primary ideal, and $M$ is a finite $S$-module.   Recall that one way to define the Hilbert-Samuel multiplicity of $M$ with respect to $J$ is as the limit,
\begin{align*}
	e\left(J, \, M\right) := \lim_{n \to \infty} \dfrac{a!}{n^a} \len{S}{\dfrac{M}{J^{n}M}}.
\end{align*}
When $S$ has positive characteristic $p > 0,$ the Hilbert-Kunz multiplicity of an $\mf{m}_S$-primary ideal is given by a similar looking limit,
\begin{align*}
	\ehkideal{J}{S} := \lim_{e \to \infty} \dfrac{1}{p^{ea}} \len{S}{\dfrac{S}{J^{\fbp{p^e}}S}},
\end{align*}
where, here, $J^{\fbp{p^e}} = (x^{p^e} \, | \, x \in J),$ denotes the Frobenius bracket power of $J$.   For details about Hilbert-Samuel multiplicity, the author recommends \cite{HSintegral} and chapter 4 of \cite{bh1998}.  For Hilbert-Kunz multiplicity the treatment in \cite{polstra2018f} and the survey article \cite{huneke2014hilbert} are excellent places to start.
\begin{remark} \label{mixedchar}
The results in section \ref{HSsection} are stated only for equicharacteristc local rings.  While the technical lemmas of section \ref{TechTools} are characteristic independent, the argument presented below utilizes construction \ref{sopDiagram}, which uses the equicharacteristic form of the Cohen structure theorem.  However, when $R$ has mixed characteristic and $I$ satisfies additional assumptions (e.g. if the residue characteristic is a parameter on $R/I$) there is an analogous construction of the desired diagram, $D,$ given a system of parameters on $R/I$.  In this way, theorem \ref{HSper} can be extended to the mixed characteristic case.  
\end{remark}
\subsection{Hilbert-Samuel Multiplicity} \label{HSsection}
\indt We begin with a quick result, which shows that forming a reduction modulo $I$ behaves extremely well in our setting.  \\
	\begin{lemma} \label{reductionsperturb}
	Suppose that $(S, \mf{m}_S)$ is a local Noetherian ring and that $I, J, K \subset S$ are ideals with $K \subset J.$  Suppose that $$K J^k + I = J^{k+1} + I,$$ for some $k \in \mathbb{N}.$  Then, for all $\unideal{f} = (f_1, \dots, f_c) = I,$ and all $\epsilon_1, \dots, \epsilon_c \in \mf{m}_S\left(I + K^{k+1}\right),$ there is an equality $$K J^{k} + \unideal{f+\epsilon} = J^{k+1} + \unideal{f+\epsilon}.$$ 
	\end{lemma}
	\begin{proof}
Indeed, $K \subset J$ so the following ideals are equal by lemma \ref{naklemma}:
\begin{align*}
	KJ^k + I = KJ^k + I + K^{k+1} = KJ^k + \unideal{f+\epsilon} + K^{k+1} = KJ^k + \unideal{f+\epsilon}
\end{align*}
and
\begin{align*}
	J^{k+1} + I = J^{k+1} + I + K^{k+1} = J^{k+1} + \unideal{f+\epsilon} + K^{k+1} = J^{k+1} + \unideal{f+\epsilon} \, .
\end{align*} 
	\end{proof} 	
\begin{obs} \label{redinclusion} Notice that, if we additionally assume $I \subset J$ in lemma \ref{reductionsperturb}, the proof automatically gives $$\unideal{f+\epsilon} \subset \unideal{f+\epsilon} + J^{k+1} = I + J^{k+1} \subset J.$$ This observation makes an appearance in the proof of theorem \ref{HSper}. 
\end{obs} \\
\indt  The next result is well known, but we have stated it in a convenient form. 
\begin{lemma} \label{minredHS}
	Suppose that $(S, \mf{m}_S, \kappa)$ is an equicharacteristic complete Notherian local ring of dimension $d$, with infinite residue field, and let $J \subset S$ be an $\mf{m}_S$-primary ideal.  Then there is a system of parameters, $x_1, \dots, x_d \in J$, for $S$ such that \\
	\begin{itemize}
		\item[(i)] $\eta = (x_1, \dots, x_d)S$ is a reduction of $J$;
		\item[(ii)] $S$ is module finite over the regular local ring $A :=\kappa [[x_1, \dots, x_d]]$ and for every finite $S$-module $V$,
		\begin{align*}
			e_S\left(J, V\right) = e_A \left(V\right) = \rank{A}{V} \, .
		\end{align*}
	\end{itemize}
\end{lemma}	
\begin{proof}
	The final equality, $e_A(V) = \rank{A}{V},$ holds since $A$ is a regular local domain, so
\begin{align*}
	e_A \left(V\right) = \rank{A}{V} e(A) = \rank{A}{V}
\end{align*}
(see 14.4 and 14.8 of \cite{matsumura1989commutative}). \\
	For the remainder, see theorems 14.13 and 14.14 of \cite{matsumura1989commutative}, or, alternatively, 8.3.7, 8.3.9 and 11.2.1 of \cite{HSintegral} for the results about multiplicity, and section 29 of \cite{matsumura1989commutative} for a discussion of the Cohen structure theorems for complete local rings.
\end{proof}	
\indt Lemma \ref{minredHS} only works when the residue field of $R$ is infinite, but this is easily circumvented. Making use of a standard construction, we are able to reduce our general setup to the case where the conditions of lemma \ref{minredHS} are satisfied.  
\begin{lemma} \label{infres}
Suppose $(S, \mf{m}_S, \kappa)$ is a local Noetherian ring, and let $Y$ be an indeterminate.  Define
	\begin{align*}
		S(Y) := S[Y]_{\mf{m}_S S[Y]}
	\end{align*} 
	and let $\widehat{S(Y)}$ denote the completion of $S(Y)$ at the maximal ideal. \\
	\indt Then, the extension $S \to \widehat{S(Y)}$ is faithfully flat, and
\begin{itemize}
	\item[(i)] the residue field of $\widehat{S(Y)}$ is infinite;
	\item[(ii)] if $V$ is a finitely generated $S$-module and $L \subset S$ is an $\mf{m}_S$-primary ideal, then
	\begin{align*}
		\dim_{S} V = \dim_{\widehat{S(Y)}} V \otimes_S \widehat{S(Y)}
	\end{align*}	
	and, further, 
	\begin{align*}
		\len{S}{V/L V} = \len{\widehat{S(Y)}}{\dfrac{V \otimes_S \widehat{S(Y)}}{L \left(V \otimes_S \widehat{S(Y)}\right)}};
	\end{align*}
	\item[(iii)] If $I = (f_1, \dots, f_c) \subset S$ are minimal generators in $S$, then $I \widehat{S(Y)} = (f_1, \dots, f_c) \widehat{S(Y)}$ are minimal generators in $\widehat{S(Y)}.$ 
\end{itemize}
\end{lemma}
\begin{proof}
	The extension $S \xhookrightarrow{} \widehat{S(Y)}$ is the composition of faithfully flat extensions. Further, the residue field of $\widehat{S(Y)}$ is isomorphic to the fraction field of $\kappa[Y]$, which is infinite. \\
	 The properties of the extension $S \to S(Y)$ are discussed in section 8.4 of \cite{HSintegral} and  on pg. 114 of \cite{matsumura1989commutative}. See section 8 of \cite{matsumura1989commutative}  for details about completion in Noetherian local rings.
\end{proof}
\indt We are ready to state the main result of the section.
\begin{thm} \label{HSper}
	Let $I$ be an ideal in an equicharacteristic Noetherian local ring, $(R, \mf{m}_R),$ and suppose that $M$ is a finite $R$-module.  If $\dim R/I \ge 1$ assume additionally that 
\begin{align*}
	\dim_R \hgy{1}{I; M} = \dim \dfrac{R}{\ann{R}{\hgy{1}{I; M}}} < \dim R/I \, .
\end{align*}
Then, for any $\mf{m}_R$-primary ideal $J \supset I,$ there is a $T \in \mathbb{N}$ such that, for all minimal generators $(f_1, \dots, f_c) = I$ and all $\epsilon_1, \dots, \epsilon_c \in \mf{m}_R^T$, 
	\begin{align*}
		e\left(\, J\ol{R}, \, \ol{M}\right) = e\left(\, J\olper{R}, \,  \olper{M}\right) \, .
	\end{align*}	
	\end{thm}
	\begin{proof}
	If $I$ happens to be $\mf{m}_R$-primary, then the stronger conclusion of corollary \ref{nakcordim0} holds. \\
\indt Otherwise, $\ol{R}$ has dimension at least $1.$  If $R$ does not have infinite residue field, or is not complete, then we apply the construction in lemma \ref{infres}.  The map $R \to \widehat{R(Y)}$ is faithfully flat and, by (iii) of lemma \ref{infres}, minimal generators for $I$ in $R$ map to minimal generators for $I\widehat{R(Y)}$ in $\widehat{R(Y)},$ so there is an isomorphism
\begin{align*}
	\hgy{1}{I; M} \otimes_R \widehat{R(Y)} \cong \hgy{1}{I\widehat{R(Y)}; M \otimes_R \widehat{R(Y)}}.
\end{align*}
So, (ii) of lemma \ref{infres} gives, 
\begin{align*}
	\dim_{\widehat{R(Y)}} \hgy{1}{I\widehat{R(Y)}; M \otimes_R \widehat{R(Y)}} = \dim_R \hgy{1}{I ; M} < \dim R/I = \dim \widehat{R(Y)}/I\widehat{R(Y)}.
\end{align*}
Thus, all of the conditions in the statement of theorem \ref{HSper} now hold over $\widehat{R(Y)}$ with $I$ replaced by $I\widehat{R(Y)}$ and $M$ replaced by $M \otimes_R \widehat{R(Y)}.$  Furthermore, by the second equality in (ii) of lemma \ref{infres} the multiplicities in question, $e\left(\, J\ol{R}, \, \ol{M}\right)$ and $e\left(\, J\olper{R}, \,  \olper{M}\right),$ do not change upon passing to $\widehat{R(Y)}$.  \\
\indt We have reduced to the case where $\dim \ol{R} \ge 1,$ and $R$ (and thus $\ol{R}$) is complete with infinite residue field.  By lemma \ref{minredHS} we may choose parameters, $x_1, \dots, x_a$ on $\ol{R}$ that generate a minimal reduction of $J\ol{R} = J/I.$  These parameters induce a commuting diagram of the form $D$, as in construction \ref{sopDiagram}, with $A = \left(R/\mf{m}\right)[[x_1, \dots, x_a]].$  By (ii) of lemma \ref{minredHS}, we have 
\begin{align*}
	e\left( \, J\ol{R}, \, \ol{M}\right) = \rank{A}{\ol{M}} \, .
\end{align*}
\indt We will abuse notation slightly, and use the same symbols $x_1, \dots, x_a$ to denote lifts of these parameters to $R$. In accordance with lemma \ref{reductionsperturb}, choose $T$ large enough so that the images of $x_1, \dots, x_a$ continue to generate a reduction of $J\olper{R}$ in $\olper{R}$ --- note that, for such $T$, $\unideal{f+\epsilon} \subset (x_1, \dots, x_a)R + I \subset J,$ as pointed out in observation \ref{redinclusion}.  We have,
\begin{align*}
	e \left(\, J\olper{R}, \, \olper{M} \right) = e_A \left( \olper{M} \right) = \rank{A}{\olper{M}}
\end{align*}
where the first equality holds by theorem 14.13 of \cite{matsumura1989commutative}, because $\mf{m}_A \olper{R}$ is a reduction of $J\olper{R}$; and the second equality is a consequence of the fact that $A$ is a regular local domain, as in the proof of lemma \ref{minredHS} (by 14.8 of \cite{matsumura1989commutative}). \\
\indt Now, by assumption, 
\begin{align*}
	\dim \hgy{1}{I; M} < \dim R/I,
\end{align*}
and so, by lemma \ref{dimH1} there must be a nzd $d \in A$ such that $d\hgy{1}{I; M} = 0.$  In particular, the conditions of corollary \ref{rankper} are satisfied and therefore, for sufficiently large $T$, we have an equality
\begin{align*}
	e \left(\, J\olper{R}, \, \olper{M} \right) = \rank{A}{\olper{M}} = \rank{A}{\ol{M}} = e\left( \, J\ol{R}, \, \ol{M}\right) \, .
\end{align*}
	\end{proof}
\subsection{Hilbert-Kunz Multiplicity}
\indt  Our next goal is to prove a more general form of the following result  (see theorem 4.3 of \cite{polstra_smirnov_2018} and corollary 3.7 of \cite{AIstab}):
\begin{thm} \label{IllyaThomasHK} \textbf{(Polstra and Smirnov)}  Let $(R, \mf{m}_R)$ be a d-dimensional $F$-finite Cohen-Macaulay local ring of prime characteristic $p>0,$ and let $I$ be an ideal generated by $c>0$ parameters.  Assume that $\widehat{R}/I\widehat{R}$ is reduced.  \\
\indt Then, for any $\mf{m}_R$-primary ideal $J \supset I$, and any $\delta > 0,$ there exists a $T \in \mathbb{N}$ such that for any minimal generators, $(f_1, \dots, f_c) = I,$ and any $\epsilon_1, \dots, \epsilon_c \in \mf{m}_R^T$,
\begin{align*}
	\left| \ehk{J\ol{R}} - \ehk{J\olper{R}} \right| < \delta \, .
\end{align*}
\end{thm}
\indt The proof in \cite{polstra_smirnov_2018}, makes use of techniques developed in section 6 of Hochster and Hueneke's classic tight closure paper \cite{HH90}.  Our theorem \ref{freeper} pairs very nicely with these methods, and we are able to modify the argument from \cite{polstra_smirnov_2018} so that it works without requiring any strong Cohen-Macaulay-ness assumptions.  To do this we need to apply the 'discriminant technique' of \cite{HH90} outside of it's original scope, and, since it is difficult to find everything needed in a single source, we record some facts and definitions about the trace and discriminant before moving on to the main results of this section.  
\subsubsection{Trace, Discriminant, and Generic \'Etale Algebras} \label{traceSec}
\indt For this discussion, $A \xhookrightarrow{} S$ is a module finite extension of Noetherian local rings, with $A$ a domain.  We will use $K$ to denote the quotient field of $A$.  Recall that $S$ is \textbf{generically \'etale} (or, equivalently, \textbf{generically separable}) over $A$, provided that $S \otimes_A K$ is \'etale as a $K$-algebra, i.e. $S \otimes_A K$ is $K$-algebra isomorphic to a finite product of fields,
\begin{align*}
	S \otimes_A K \cong \prod_{i=1}^n L_i,
\end{align*}    
with each $L_i$ a \textit{separable} extension of $K$.  \\
\indt The kernel of the map $S \to S \otimes_A  K$ is the ideal of $A$-torsion elements
\begin{align*}
	T_A(S) := \{r \in S \, | \, ar = 0, \text{ for some } 0 \neq a \in A \},
\end{align*}  
$S':=S/T_A(S)$ is torsion-free as an $A$-module, and there is a commutative diagram of $A$-algebras,
$$ C:
	\begin{tikzcd}    
			 S \arrow[r] \arrow[d] & S' \arrow[d] \\
		 	 S \otimes_A K \arrow[r, "\sim"]	 & S' \otimes_A K
 	\end{tikzcd}
$$
where the bottom map is a $K$-algebra isomorphism
\begin{align*}
	S \otimes_A K \cong S' \otimes_A K.
\end{align*}
Notice, in particular, that $S$ is generically \'etale over $A$ if and only if $S'$ is. \\
\indt Now, $S \otimes_A K$ is a finite $K$-vector space, and given any $x \in S \otimes_A K,$ multiplication by $x$ is a $K$-linear endomorphism
\begin{align*}
	 S \otimes_A K \to S \otimes_A K.
\end{align*}
Denote the trace of this map by $\trc{S \otimes_A K}{K}{x}.$  This trace is $K$-linear in $x$, and therefore we have defined a $K$-linear map
\begin{align*}
	\mathrm{trace}_{S\otimes_A K/K} : S \otimes_A K \to K.
\end{align*}
This, in turn, determines a $K$-bilinear form on the algebra $S \otimes_A K$ given by
\begin{align*}
	\tray{x}{y} := \trc{S \otimes_A K}{K}{xy}.
\end{align*}
\begin{prop} \label{nondegetale}
  $S \otimes_A K$ is an \'etale $K$-algebra if and only if $\mathrm{Tr}$ is non-degenerate.
\end{prop}
\begin{proof}
 If $S \otimes_A K \cong \prod_{i=1}^n L_i$ is isomorphic to a finite product of finite separable extensions of $K,$ the trace form is non-degenerate since its restrictions to each $L_i$ are non-degenerate.  For the other direction, note that the trace of a nilpotent element vanishes, and so if $\mathrm{Tr}$ is non-degenerate then $S \otimes_A K$ is reduced.  This implies that $S/T_A(S) = S'$ is reduced, and therefore $S \otimes_A K \cong S' \otimes_A K$ is a finite product of finite field extensions of $K$.  Now the result follows from the linearity of the trace and the fact that a finite field extension $L/K$ is separable if and only the trace form of $L$ over $K$ is non-degenerate. \\  
\indt See the proof of the lemma in \cite[\href{https://stacks.math.columbia.edu/tag/0BVH}{Section 0BVH}]{stacks-project}. 
\end{proof}
\indt Given a basis $\ul{b} = b_1, \dots, b_n,$ for $S \otimes_A K$ as a $K$ vector space, the \textbf{discriminant} of $\mathrm{Tr}$ with respect to $\ul{b}$ is defined to be
\begin{align*}
	\mc{D}\left(\mathrm{Tr}, \ul{b}\right) := \det \left( \tray{b_i}{b_j} \right).
\end{align*}
If $\ul{b}' = b_1', \dots, b_n'$ is another basis for $S \otimes_A K,$ and $A$ denotes the matrix transforming from $\ul{b}$ to $\ul{b}'$, a direct calculation shows that
\begin{align*}
	\mc{D}\left(\mathrm{Tr}, \ul{b}'\right) = (\det A)^2 \mc{D}\left(\mathrm{Tr}, \ul{b}\right) \, .
\end{align*}
From this identity it is clear that the vanishing or non-vanishing of the discriminant of $\mathrm{Tr}$ does not depend on the choice of basis, and it is straightforward to show that a bilinear form is non-degenerate if and only if its discriminants are nonzero (see e.g. the discusion on the bottom of pg. 232 of \cite{jacobsonbasic2}).  Therefore, proposition \ref{nondegetale} may be reformulated in terms of the discriminant.
\begin{prop} \label{discetale}
The following are equivalent:
\begin{itemize}
	\item[(i)] $S \otimes_A K$ is an \'etale $K$-algebra;
	\item[(ii)] $\mc{D}\left(\mathrm{Tr}, \ul{b}\right) \neq 0$ for every $K$ basis, $\ul{b}$ of $S \otimes_A K$; 
	\item[(iii)] $\mc{D}\left(\mathrm{Tr}, \ul{b}\right) \neq 0$ for some $K$ basis, $\ul{b},$ of $S \otimes_A K.$
\end{itemize}
\end{prop}
\indt Composing $\mathrm{trace}_{S \otimes_A K / K}$ with the map $S \to S \otimes_A K$ gives an $A$-linear map $S \to K.$  When $A$ is normal, the image of this composition actually lands in $A$.
\begin{lemma} \label{tracemap}
	Let $A \xhookrightarrow{} S$ be a module finite map of Noetherian local rings, with $A$ a normal domain.  Let $K$ be the quotient field of $A$.  For all $r \in S,$ the trace of the multiplication by $r$ map
\begin{align*}
	S \otimes_A K \xrightarrow{r \cdot } S \otimes_A K
\end{align*}  
belongs to $A$.
\end{lemma}
\begin{proof}
  The argument on page 200 of \cite{Hfoundations} establishes the result when $S$ is torsion-free over $A$.  Referring to diagram C, above, we see that multiplication by any $r \in S$ on $S \otimes_A K$ agrees with multiplication by it's image in $S' = S/T_A(S)$  on $S' \otimes_A K$ --- so the torsion-free case implies the result in general.   
\end{proof}
\indt Therefore, when $A$ is normal, the trace on $S \otimes_A K$ determines an $A$-linear map, $\mathrm{trace}_{S/A}: S \to A,$ called the \textbf{trace} of $S$ over $A.$  Moreover, if $r_1, \dots, r_n$ are elements in $S$ whose images, $r_1 \otimes 1, \dots r_n \otimes 1 = \ul{r\otimes 1}$ constitute a $K$ basis in $S \otimes_A K,$ lemma \ref{tracemap} implies that the discriminant $\mc{D}\left(\mathrm{Tr}, \ul{r \otimes 1}\right)$ is an element of $A$.
\subsubsection{Perturbing the Trace and Discriminant} \label{subSecper}
\indt Returning to the theme of this paper, suppose that $(R, \mf{m}_R)$ and $(A, \mf{m}_A)$ are complete Noetherian local rings, that $I \subset R$ is an ideal and that $R,$ $A$ and $\ol{R} = R/I$ are arranged as in diagram $D$.  In addition, assume that $A$ is a normal domain of positive dimension, and that there is a nonzero element $d \in A$, such that 
\begin{align*}
	d \hgy{1}{I; R} = 0.
\end{align*}
Denote the quotient field of $A$ by $K$. \\
\indt Choose elements, $r_1, \dots, r_n \in R$ that span a free $A$-submodule of maximum rank in $\ol{R},$ and write $$N := Ar_1 + \dots + Ar_n.$$  Note that $\ol{R}$ contains a copy of $A$ so it has positive $A$-rank.  So, $N \cong A^{\oplus n}$ is a free $A$-module, and there is an exact sequence of finite $A$-modules
\begin{align*}
	0 \to N \to \ol{R} \to Q \to 0
\end{align*}
where $Q$ is $A$-torsion. \\
\indt The conditions of theorem \ref{freeper} and corollary \ref{rankper} are satisfied, so we may fix a $T \gg 0$ such that the image of $N$ in $\olper{R}$ is a free $A$-module of maximal rank, for all minimal generators $\unideal{f} = I,$ and any $\ul{\epsilon} \in \mf{m}_R^T.$  For these $\unideal{f+\epsilon}$, there are exact sequences of finite $A$-modules
\begin{align*}
	0 \to N \to \olper{R} \to \per{Q} \to 0
\end{align*}
with $\per{Q}$ torsion over $A$.  Tensoring with $K$, we conclude that for each of these $\unideal{f+\epsilon}$ there are $K$-module isomorphisms
\begin{align*}
	N \otimes_A K \cong \olper{R} \otimes_A K,
\end{align*} 
and we have a commuting diagrams
$$	\begin{tikzcd}    
			 N \arrow[r] \arrow[d] & \olper{R} \arrow[d] \\
		 	 N \otimes_A K \arrow[r, "\sim"]	 & \olper{R} \otimes_A K
 	\end{tikzcd}
$$
Of course, this means that these $\olper{R} \otimes_A K$ are all isomorphic as $K$ vector spaces, though they are not necessarily isomorphic as $K$-algebras --- the multiplication depends on multiplication in the $\olper{R},$ which may well be non-isomorphic rings, even for arbitrarily large $T$. \\
\indt It follows from this that the generators, $\ul{r} = r_1, \dots, r_n$ of $N$, simultaneously map to $K$-bases for all of the algebras $\olper{R} \otimes_A K.$  Next, we are going to use this simultaneous basis to compute the $A$-linear trace maps $\mathrm{trace}_{\olper{R} /A}: \olper{R} \to A,$ showing that they are 'close together' as maps in $\hm{A}{R}{A}.$ \\
\indt Recall, it was established in lemma \ref{anntors}, that we will have $\mu_A(\ol{R}) = \mu_A(\olper{R})$ when $\mf{m}_R^T \subset \mf{m}_R\left(I +\mf{m}_AR\right)$.  Following the notation of lemma \ref{anntors}, we will let $\ul{N}$ and $\ulper{N}$ denote the images of $N$ in $\ol{R}$ and $\olper{R}$, respectively.  Composing the $\mathrm{trace}_{\olper{R}/A}$ maps with the projections $R \to \olper{R}$ produces $A$-linear maps in $\hm{A}{R}{A}$ --- we will continue to denote these maps by $\mathrm{trace}_{\olper{R}/A}.$
\begin{lemma} \label{traceper}
Suppose that $T \gg 0$ is large enough that the discussion in the previous few paragraphs holds, and fix any $0 \neq c \in \ann{A}{\ol{R}/\ul{N}}$.  Set $m = \mu_A(\ol{R}),$ and increase $T$, if necessary, to ensure that
$$\mf{m}_R^T \subset \mf{m}_R \left( I + \mf{m}_A^H R\right),$$ where $H > 2t',$ with $$t' = \max \left(\AR{\mf{m}_A}{c^{2m}A \subset A}{A}, \, \AR{\mf{m}_A}{\ul{N} \subset \ol{R}}{A} \right).$$ 
Then, for any $r \in R,$ any minimal generators $\unideal{f} = I,$ and all $\epsilon_1, \dots, \epsilon_c \in \mf{m}_R^T,$ 
\begin{align*}
	\trc{\ol{R}}{A}{r} - \trc{\olper{R}}{A}{r} \in \mf{m}_A^{H-2t'} \, .
\end{align*}
\end{lemma}
\begin{proof}
\indt Fix any such $\unideal{f+\epsilon},$ and choose, as in lemma \ref{anntors}, a $b \in \ann{A}{\olper{R}/\ulper{N}},$ such that $$b - c^m \in \mf{m}_A^H.$$  By construction, we have,
\begin{align*}
	bR &\subset N + \unideal{f+\epsilon} \\
	c^m R &\subset N + I \, .
\end{align*} 
Thus, given any fixed $r \in R,$ there are $\alpha_{ij}, \beta_{ij} \in A,$ $i,j = 1, \dots, n,$ such that
\begin{align*}
	bc^mrr_i &= \alpha_{i1}r_1 + \dots + \alpha_{in}r_n + F_i \\
			 &= \beta_{i1}r_1 + \dots + \beta_{in} r_n + G_i,
\end{align*}
with $F_1, \dots, F_n \in I,$ and $G_1, \dots, G_n \in \unideal{f+\epsilon}$ --- recall, here, $\ul{r} = r_1, \dots, r_n$ is our basis for $N$ as an $A$-module. \\
\indt Mapping this equality to $\ol{R}$, and using $\ol{r_i}$ to denote the image of $r_i$ in $\ol{R},$ we conclude that
\begin{align*}
	\left(\alpha_{i1}-\beta_{i1}\right)\ol{r_1} + \dots + \left(\alpha_{in}-\beta_{in}\right)\ol{r}_n = \ol{G_i}.
\end{align*}
Each of these $\ol{G_i} \in \unideal{f+\epsilon}\ol{R} \subset \mf{m}_A^H \ol{R},$ and so 
\begin{align*}
	\left(\alpha_{i1}-\beta_{i1}\right)\ol{r_1} + \dots + \left(\alpha_{in}-\beta_{in}\right)\ol{r}_n \in \mf{m}_A^H \ol{R} \cap \ul{N} \subset \mf{m}_A^{H-t'} \ul{N}.
\end{align*}
Since $\ul{N}$ is free over $A$ with basis $\ol{r_1}, \dots, \ol{r_n}$, we conclude that $$\alpha_{ij} - \beta_{ij} \in \mf{m}_A^{H-t'},$$ for each $i,j = 1, \dots, n.$ \\
\indt Now, the trace is $A$-linear, so we have
\begin{align*}
	bc^m \left(\trc{\olper{R}}{A}{r}\right) &= \trc{\olper{R}}{A}{bc^mr} \\
	                                       &= \trc{\olper{R}\otimes_A K}{K}{bc^mr} \\
	                                       &= \beta_{11} + \beta_{22} + \dots + \beta_{nn},
\end{align*}
where the last equality follows from the fact, established above, that the images of $r_1, \dots, r_n$ constitute a basis in $\olper{R} \otimes_A K.$  The same reasoning applies to $bc^m\trc{\ol{R}}{A}{r},$ and so we have
\begin{align*}
	bc^m \left(\trc{\ol{R}}{A}{r}\right) = \alpha_{11} + \dots + \alpha_{nn}.
\end{align*}
\indt Therefore, we have established that 
\begin{align*}
	bc^m\left(\trc{\ol{R}}{A}{r} - \trc{\olper{R}}{A}{r} \right) \in \mf{m}_A^{H-t'},
\end{align*}
and, recalling that $b - c^m \in \mf{m}_A^H,$ we have
\begin{align*}
	c^{2m} \left(\trc{\ol{R}}{A}{r} - \trc{\olper{R}}{A}{r} \right) \in \mf{m}_A^{H-t'} \cap \left(c^{2m}A\right) \subset \mf{m}_A^{H-2t'}\left(c^{2m}A\right).
\end{align*}
Since $A$ is a domain, and $c \neq 0,$ we conclude that 
\begin{align*}
	\trc{\ol{R}}{A}{r} - \trc{\olper{R}}{A}{r} \in \mf{m}_A^{H-2t'},
\end{align*}
and the claim is proved.
\end{proof}
\indt Continuing with the setup above, assume now that $T$ is chosen as in lemma \ref{traceper}.  For each minimal $(f_1, f_2, \dots, f_c) = I,$ and any $\epsilon_1, \dots, \epsilon_c \in \mf{m}_R^T,$ let $\mathrm{Tr}_{\unideal{f+\epsilon}}$ denote the bilinear trace form on $\olper{R} \otimes_A K,$ and let $\mathrm{Tr}_I$ denote the corresponding trace form on $\ol{R} \otimes_A K.$  Define discriminants over $A$, as follows,
\begin{align*}
	\disc{A}{\ol{R}} &:= \mc{D}\left(\mathrm{Tr}_I, \, \ul{r}\right) = \det \left(\trc{\ol{R}}{A}{r_i r_j}\right) \\
	\disc{A}{\olper{R}} &:= \mc{D}\left(\mathrm{Tr}_{\unideal{f+\epsilon}}, \, \ul{r}\right) = \det \left(\trc{\olper{R}}{A}{r_i r_j} \right).
\end{align*}
\indt We have just proved, in lemma \ref{traceper},that the matricies
\begin{align*}
	\left(\trc{\ol{R}}{A}{r_i r_j}\right) \,\,\, \text{   and   } \,\,\, \left(\trc{\olper{R}}{A}{r_i r_j}\right)	
\end{align*}
are equal modulo $\mf{m}_A^{H-2t'}.$  It follows that their determinants must also be equal in $A/\mf{m}_A^{H-2t'}.$
\begin{cor} \label{Discper}
	Under the conditions of lemma \ref{traceper}, we have
\begin{align*}
	\disc{A}{\ol{R}} - \disc{A}{\olper{R}} \in \mf{m}_A^{H-2t'} \, .
\end{align*}
\end{cor}
\begin{obs} \label{DisTorper}
Note that, if we let $c \in \ann{A}{\ol{R}/\ul{N}}$ and $b \in \ann{A}{\olper{R}/\ulper{N}}$ be as in the proof of lemma \ref{traceper} --- so that $c^m - b \in \mf{m}_A^H$ --- the arguments above show that we have,
\begin{align*}
	c^m\disc{A}{\ol{R}} - b\disc{A}{\olper{R}} \in \mf{m}_A^{H-2t'}.
\end{align*}
\end{obs}
\subsubsection{Perturbing the Hilbert-Kunz Multiplicity} \label{HKsection}
\indt The next result, due to Smirnov, is a convenient generalization of the results from section 6 of \cite{HH90}. \\
\begin{lemma} \label{smircor23} \textbf{(see Corollary 2.3 of \cite{I2019semicont})} 
Suppose that $A$ is a normal, Noetherian local ring of characteristic $p>0,$ and $A \xhookrightarrow{} S$ is a module finite local extension which is generically \'etale. \\
\indt Let $\ul{s} = s_1, \dots, s_n \in S$ be elements in $S$ that span a free $A$-module of maximal rank,
\begin{align*}
	N: = As_1 + \dots, +As_n.
\end{align*}
Let $K$ denote the quotient field of $A$, write $\mathrm{Tr}$ for the trace form of $S$ over $A$, and set
\begin{align*}
	\disc{A}{S} := \mc{D}\left(\mathrm{Tr}, \ul{s}\right).
\end{align*}
\indt Suppose that $a \in \ann{A}{S/N},$ and let $F: S \to S$ denote the Frobenius endomorphism.  Then, there are exact sequences of $S$-modules,
\begin{align*}
	A^{1/p} \otimes_A S \to F_{*}S \to Q_1 &\to 0 \\
	F_{*}S \to A^{1/p} \otimes_A S \to Q_2 &\to 0
\end{align*}
such that $a\disc{A}{S}Q_1 = a\disc{A}{S}Q_2 = 0.$
\end{lemma}
\indt We are going to use this result to uniformly control the growth of the Hilbert-Kunz functions of the $\olper{R}$.  In order to do so, we need these rings to be generically \'etale over $A.$ 
\begin{prop} \label{genetaleper}
Suppose $(R, \mf{m}_R)$ and $(A, \mf{m})$ are complete Noetherian local rings, that $I \subset R$ is an ideal and that $R, A$ and $\ol{R} = R/I$ are related as in a diagram of the form $D$.  Suppose further that $A$ is a normal domain of positive dimension, that the finite extension $A \xhookrightarrow{} \ol{R}$ is generically \'etale, and that there is a $0 \neq d \in A$ such that
\begin{align*}
	d \hgy{1}{I; R} = 0.
\end{align*}
\indt Then, there is a $T \gg 0$ such that for all minimal generators $(f_1, \dots, f_c) = I$ and any $\epsilon_1, \dots, \epsilon_c \in \mf{m}_R^T,$ the composition $A \xhookrightarrow{} R \to \olper{R}$ is a generically \'etale module finite extension, 
$$A \xhookrightarrow{} \olper{R}.$$
\end{prop} 
\begin{proof}
	This follows from proposition \ref{discetale}, corollary \ref{Discper}, and observation \ref{injmap}.
\end{proof}
\indt We also make use of a slightly non-standard form of the Cohen-Gabber structure theorem for complete Noetherian local rings.  This result is well known to the experts, but it is difficult to find a reference.
\begin{lemma} \label{CGlemma}
	Suppose that $(S, \mf{m}_S, \kappa)$ is a complete equicharacteristic Noetherian local ring of dimension $\dim S = a.$  Suppose that $S$ is equidimensional and satisfies $\dim_{S} \nilrad{S} < a$. \\
\indt Then there exists an s.o.p. $x_1, \dots, x_a \in \mf{m}_S$ and coefficient field $\kappa \xhookrightarrow{} S,$ such that the induced map 
	\begin{align*}
		\kappa[[x_1, \dots, x_a]] \hookrightarrow S
	\end{align*} 
	is module finite and generically \'etale. 
\end{lemma}
\begin{proof}
Write $B = S/\nilrad{S},$ and note that $\dim B = \dim S$ and $B$ is both equidimensional and reduced.  By the Cohen-Gabber structure theorem (see \cite{kurano2015elementary}) there is a module finite, generically \'etale extension
\begin{align*}
	A \hookrightarrow B,
\end{align*}
with $A$ a power series ring over $\kappa,$ in a system of parameters on $B$. \\
\indt This map lifts to a module finite extension, $A \hookrightarrow S$, such that the composition $A \hookrightarrow S \twoheadrightarrow B$ agrees with
\begin{align*}
	A \hookrightarrow B.
\end{align*} 
\indt By assumption $\ann{S}{\nilrad{S}}$ contains a parameter, and therefore $A \cap \ann{S}{\nilrad{S}} \neq 0$ --- i.e. $\nilrad{S}$ is $A$-torsion, so that there is $K$-algebra isomorphism $$S \otimes_A K \cong B \otimes_A K.$$  Since $B \otimes_A K$ is \'etale over $A$, $S \otimes_A K$ is as well.  
\end{proof}
\indt Now the proof begins in earnest.  Following \cite{polstra_smirnov_2018} and \cite{AIstab} we first show how to uniformly control the growth of the Hilbert-Kunz functions of interest. \\
\begin{prop} \label{UniformBounds}
Suppose that $(R, \mf{m}_R, \kappa)$ is a Noetherian local ring of characteristic $p > 0$ such that $\kpdeg < \infty$, and that $I \subset R$ is an ideal with $a = \dim R/I \ge 1.$   Assume the following three conditions hold:
\begin{itemize}
\item[(i)] $\widehat{R}/I\widehat{R}$ is equidimensional;
\item[(ii)] $\dim_{R/I} \hgy{1}{I; R} < a$;
\item[(iii)]  $\dim_{\widehat{R}/I\widehat{R}} \nilrad{\widehat{R}/I\widehat{R}} < a \, .$ 
\end{itemize}
\indt Then, for any $e_0 \ge 0,$ there is a constant, $C = C_{e_0} > 0,$ and an integer, $T \gg 0,$ such that, for every ideal $J \subset R$ with $\mf{m}_R ^{\fbp{p^{e_0}}} \subset J,$ and all $e \ge 1$, 
\begin{align*}
	\Biggl|\len{R}{\dfrac{R}{\unideal{f+\epsilon} + J^{\fbp{p^{e+1}}}}} - p^a\len{R}{\dfrac{R}{\unideal{f+\epsilon} + J^{\fbp{p^{e}}}}} \Biggr| \le Cp^{e(a-1)}
\end{align*} 
for all minimal generators $(f_1, \dots, f_c) = I,$ and every $\epsilon_1, \dots, \epsilon_c \in \mf{m}_R^T.$
\end{prop}
\begin{remark} \label{altStatement}  Before launching into the proof of proposition \ref{UniformBounds}, we make a couple of notes about the conditions (i), (ii) and (iii) appearing in the statement. \\ 
\indt According to the discussion in section \ref{techcon}, in particular remark \ref{equidim} and proposition \ref{dimH1altsupp}, conditions (i) and (ii) in proposition \ref{UniformBounds} force $I$ to be a parameter ideal, i.e. $\mu_R(I) = \height{I}.$  \\
\indt Moreover, again by proposition \ref{dimH1altsupp}, (ii) is equivalent to the requirement that the minimal generators of $I$ are 'generically a regular sequence,' i.e. they form a regular sequence in $R_{\mf{q}}$ for every $\mf{q} \in \minh{R/I},$ and since $I$ is a parameter ideal, this is equivalent to $R_{\mf{q}}$ being Cohen-Macaulay for every $\mf{q} \in \minh{R/I}$ (see remark \ref{dimH1_converseFilterReg_Reg} (iv)). \\
\indt It is also clear that (iii) is equivalent to the condition that  $\left(\widehat{R}/I\widehat{R}\right)_{\mf{q}}$ is reduced for every $\mf{q} \in \minh{\widehat{R}/I\widehat{R}}.$
\end{remark}
\begin{proof}
Condition (ii) is unaffected by completion, neither are the lengths involved, so we assume, without any loss of generality, that $R$ is complete.  Write $\ol{R} = R/I.$  Note that we are assuming $R$ is complete with $F$-finite residue field --- it follows that $R$ itself is $F$-finite. \\
\indt Conditions (i) and (iii) hold, so lemma \ref{CGlemma} applies to $\ol{R}.$  Choose parameters $\ul{x} = x_1, \dots, x_a \in \ol{R},$ such that the associated finite extension,
\begin{align*}
	A := \kappa[[x_1, \dots, x_a]] \xhookrightarrow{} \ol{R},
\end{align*}
is generically \'etale. \\
\indt As in construction \ref{sopDiagram}, $R$, $A$ and $\ol{R}$ fit into a commutative diagram of local rings of the familiar form, 
$$D :\begin{tikzcd}[row sep=large, column sep=large]    
			& R \arrow[d, twoheadrightarrow] \\
		 A \arrow[r, hookrightarrow, "finite", swap] \arrow[ru, hookrightarrow]	 & \ol{R}
 \end{tikzcd}$$
\indt By lemma \ref{dimH1}, we may choose a nonzero $d \in A$ such that $d\hgy{1}{I; R} = 0.$   Next, pick $\ul{r} = r_1, \dots, r_n \in R$, as in section \ref{subSecper}, which span a free $A$-submodule in $\ol{R}$ of maximal rank.  As before, set
\begin{align*}
	N := Ar_1 + \dots + Ar_n,
\end{align*}
and let $\ul{N}$ denote the image of $N$ in $\ol{R}.$  Set $$\disc{A}{\ol{R}} := \mc{D}\left(\mathrm{Tr}_I, \ul{r}\right),$$  and recall that $\disc{A}{\ol{R}} \neq 0,$ by lemma \ref{discetale}.  Fix a nonzero $c \in \ann{A}{\ol{R}/\ul{N}}$ which is a non-unit, and write $m = \mu_A(\ol{R}).$  \\
\indt These choices ensure that $$g := c^m \disc{A}{\ol{R}} \in A$$ is a regular element contained in $\mf{m}_A$, so that it is a parameter on $A$.  Fix a full system of parameters for $A$, that extends $g,$ $$g , h_1, \dots, h_{a-1} = g, \ul{h} \in \mf{m}_A.$$ \\
\indt Next, set $t = \max \left(t_1, t_2, t_3, t_4\right)$ and $t' = \max \left(t_4, t_5\right)$ where these $t_i$ are the following Artin-Rees numbers from the proofs of lemmas \ref{mainlemma} and \ref{traceper}: 
\begin{align*}
	t_1 &= \AR{\mf{m}_A R}{IM \subset M}{R}, \\
	t_2 &= \AR{\mf{m}_A R}{ \knl{\partial^{M}_{1, I}} \subset M^{\oplus c}}{R}, \\
	t_3 &= \AR{\mf{m}_A}{d\ul{N} \subset \ol{M}}{A}, \\
	t_4 &= \AR{\mf{m}_A}{\ul{N} \subset \ol{M}}{A}, \\
	t_5 &= \AR{\mf{m}_A}{c^{2m}\ul{N} \subset \ol{M}}{A}.
\end{align*}
\indt Let $H \in \mathbb{N}$ be large enough that the following conditions are satisfied
\begin{itemize}
\item[(a)] $H \ge 3t + 1,$
\item[(b)] $\mf{m}_A ^{H-2t'} \subset \mf{m}_A \left( gA + \unideal{h}^{\fbp{p^{e_0}}}A \right).$
\end{itemize}
\indt Since (b) holds, we certainly have $H > 2t',$ so that the conditions of lemmas \ref{mainlemma} and \ref{traceper} are all met.  With this choice of $H$, let $T \in \mathbb{N}$ be any integer such that $$\mf{m}_R^T \subset \mf{m}_R \left( I + \mf{m}_A^H R \right).$$ \\
\indt Next we show that, for this $T$, there is a uniform constant, $C,$ such that the bound in the statement holds.  With this in mind, let $(f_1, \dots, f_c) = I$ be minimal generators and $\epsilon_1, \dots, \epsilon_c \in \mf{m}_R^T$ any fixed elements in $\mf{m}_R^T.$  The conditions of theorem \ref{freeper} and lemma \ref{rankper} are satisfied, so the image, $\ulper{N}$, of $N$ in $\olper{R}$ is a free $A$-submodule of $\olper{R}$ of maximal rank.  In accordance with lemma \ref{anntors}, we may choose $b \in \ann{A}{\olper{R}/\ulper{N}}$ such that $$ c^m - b \in \mf{m}_A^H,$$ and set, as in section \ref{subSecper},  $$\disc{A}{\olper{R}} := \mc{D}\left(\mathrm{Tr}_{\unideal{f+\epsilon}}, \ul{r}\right).$$  Define $$g' := b\disc{A}{\olper{R}},$$ and note that by corollary \ref{Discper}, we have,
\begin{align*}
	g - g' = c^m\disc{A}{\ol{R}} -b\disc{A}{\olper{R}} \in \mf{m}_A^{H-2t'} \subset \mf{m}_A\left( gA + \unideal{h}^{\fbp{p^{e_0}}}A \right),	
\end{align*}
since $H$ was chosen to satisfy condition (b), above.  Applying lemma \ref{naklemma}, this time in $A$, we see that there is an identity of ideals, 
\begin{align} \label{ghideals} 
	gA + \unideal{h}^{\fbp{p^{e_0}}}A = g'A + \unideal{h}^{\fbp{p^{e_0}}}A \, .
\end{align}
Noting that we also have,
\begin{align*}
	\epsilon_1, \dots, \epsilon_c \in \mf{m}_R^T \subset \mf{m}_R \left(I + \mf{m}_A^HR\right) \subset \mf{m}_R \left(I + \left(gA + \unideal{h}^{\fbp{p^{e_0}}}A\right)R\right),
\end{align*}
 it follows from equality \ref{ghideals}, and another application of lemma \ref{naklemma}, that  
\begin{align} 
	\unideal{f+\epsilon} + \left(g'A + \unideal{h}^{\fbp{p^{e_0}}}A\right)R = \unideal{f+\epsilon} + \left(gA + \unideal{h}^{\fbp{p^{e_0}}}A\right)R = I + \left(gA + \unideal{h}^{\fbp{p^{e_0}}}A\right)R \, . \label{ghIideals}
\end{align}
\indt Now, the above considerations certainly imply that $g' = b \disc{A}{\olper{R}} \neq 0,$ so the module finite extension $A \xhookrightarrow{} \olper{R}$ is generically \'etale by lemma \ref{discetale}, and therefore lemma \ref{smircor23} provides exact sequences of $\olper{R}$-modules
\begin{align} 
	A^{1/p} \otimes_A \olper{R} \to F_{*}\olper{R} \to Q_1 &\to 0 \label{exactQ1} \\
	F_{*}\olper{R} \to A^{1/p} \otimes_A \olper{R} \to Q_2 &\to 0  \label{exactQ2}
\end{align}
with $g'Q_1 = g'Q_2 = 0.$ \\
\indt $A$ is a regular local ring of dimension $a$, so, by Kunz's theorem, $A^{1/p}$ is a free $A$-module of rank $\kpdeg p^a$ --- recall that $\kpdeg < \infty.$  Incorporating this, tensoring these exact sequences with $-\otimes_R \dfrac{R}{J^{\fbp{p^e}}},$ and taking lengths, we get,
\begin{align*}
	-\len{R}{\dfrac{Q_2}{J^{\fbp{p^e}}Q_2}} \le \len{R}{F_{*}\olper{R} \otimes_R \dfrac{R}{J^{\fbp{p^e}}}} - \kpdeg p^a \len{R}{\dfrac{R}{\unideal{f+\epsilon} + J^{\fbp{p^e}}}} \le \len{R}{\dfrac{Q_1}{J^{\fbp{p^e}}Q_1}} \, .
\end{align*}
There is an $R$-module isomorphism, $$F_{*}\olper{R} \otimes_R \dfrac{R}{J^{\fbp{p^e}}} \cong \dfrac{F_{*}R}{F_{*}\left(\unideal{f+\epsilon} + J^{\fbp{p^{e+1}}}\right)},$$ and so the $R$-length can be computed as follows
\begin{align*}
	\len{R}{\dfrac{F_{*}R}{F_{*}\left(\unideal{f+\epsilon} + J^{\fbp{p^{e+1}}}\right)}} = \kpdeg \len{F_{*}R}{\dfrac{F_{*}R}{F_{*}\left(\unideal{f+\epsilon} + J^{\fbp{p^{e+1}}}\right)}} = \kpdeg \len{R}{\dfrac{R}{\unideal{f+\epsilon}+ J^{\fbp{p^{e+1}}}}} \, .
\end{align*}	
Therefore our inequality has become,
\begin{align}
	-\len{R}{\dfrac{Q_2}{J^{\fbp{p^e}}Q_2}} \le \kpdeg \left( \len{R}{\dfrac{R}{\unideal{f+\epsilon}+ J^{\fbp{p^{e+1}}}}} - p^a \len{R}{\dfrac{R}{\unideal{f+\epsilon} + J^{\fbp{p^e}}}} \right) \le \len{R}{\dfrac{Q_1}{J^{\fbp{p^e}}Q_1}} \, . \label{ineqQ}
\end{align}
Our goal is to bound the lengths appearing on the far sides of this. \\
\indt  From the exact sequences (\ref{exactQ1}) and (\ref{exactQ2}), we have surjections,
\begin{align*}
	F_{*}\olper{R} &\twoheadrightarrow Q_1 \\
	A^{1/p} \otimes_A \olper{R} &\twoheadrightarrow Q_2 
\end{align*}
Recalling that both $Q_1$ and $Q_2$ are annihilated by $g'$, and utilizing the isomorphism from Kunz's theorem, $A^{1/p} \cong A^{\oplus \kpdeg p^a},$ there are surjections
\begin{align*}
	\left(\dfrac{R}{\unideal{f+\epsilon} + g'R}\right)^{\oplus v_1} \twoheadrightarrow Q_1 \,\,\,\,\,\,\, \text{   and    } \,\,\,\,\,\,\,\, 	\left(\dfrac{R}{\unideal{f+\epsilon} + g'R}\right)^{\oplus v_2} \twoheadrightarrow Q_2
\end{align*}
where $v_1 = \mu_{\olper{R}}\left(F_{*}\olper{R}\right)$ and $v_2 = \kpdeg p^a.$ \\
\indt Tensoring with $-\otimes_R \dfrac{R}{J^{\fbp{p^e}}}$ and taking lengths, we get inequalities
\begin{align}
	\len{R}{\dfrac{Q_i}{J^{\fbp{p^e}}Q_i}} \le v_i \len{R}{\dfrac{R}{\unideal{f+\epsilon} + g'R + J^{\fbp{p^e}}}} \label{Qiineq}
\end{align}
for $i = 1, 2.$ \\
\indt The $v_2$ term here is already independent of the choice of $\unideal{f+\epsilon};$ so we turn our attention to finding a uniform bound for the remaining terms in (\ref{Qiineq}).   For the $v_1$ term we have,
\begin{align*}
	v_1 = \mu_{\olper{R}}\left(F_{*}\olper{R}\right) = \len{R}{\dfrac{F_{*}R}{F_{*}\left(\unideal{f+\epsilon}+\mf{m}_R^{\fbp{p}}\right)}} = \kpdeg \len{R}{\dfrac{R}{\unideal{f+\epsilon}+\mf{m}_R^{\fbp{p}}}} \, .
\end{align*}
There is an inclusion $\mf{m}_A^{\fbp{p}}R \subset \mf{m}_R^{\fbp{p}},$ and so we continue this string of inequalities,
\begin{align*}
	v_1  = \kpdeg \len{R}{\dfrac{R}{\unideal{f+\epsilon}+\mf{m}_R^{\fbp{p}}}} \le \kpdeg \len{R}{\dfrac{R}{\unideal{f+\epsilon}+\mf{m}_A^{\fbp{p}}R}} \le \kpdeg p^a \len{R}{\dfrac{R}{\unideal{f+\epsilon} + \mf{m}_A R}} \, .
\end{align*}
where the last inequality comes from lemma \ref{filtration}, and the fact that $\mf{m}_A$ is generated by $a$ elements.  By lemma \ref{naklemma}, we have $\unideal{f+\epsilon} + \mf{m}_AR = I + \mf{m}_A R,$ and so we have shown that
\begin{align*}
	v_1 \le \kpdeg p^a \len{R}{\dfrac{R}{I + \mf{m}_A R}} = \kpdeg p^a m,
\end{align*}
where, as above, $m = \mu_A(\ol{R}).$ \\
\indt Now we bound the first term in \ref{Qiineq}.  By assumption, there is an inclusion, $\mf{m}_R^{\fbp{p^{e_0}}} \subset J,$ and, certainly $\unideal{h}R \subset \mf{m}_R,$ so we have
\begin{align*}
	\len{R}{\dfrac{R}{\unideal{f+\epsilon} + g'R + J^{\fbp{p^e}}}} \le \len{R}{\dfrac{R}{\unideal{f+\epsilon} + g'R + \mf{m}_R^{\fbp{p^{e_0+e}}}}} \le \len{R}{\dfrac{R}{\unideal{f+\epsilon} + g'R + \unideal{h}^{\fbp{p^{e_0+e}}}}} \, .
\end{align*}
The ideal $\unideal{h}^{\fbp{p^{e_0}}}R = (h_1^{p^{e_0}}, \dots, h_{a-1}^{p^{e_0}})R$ is generated by $a-1$ elements, and so applying lemma \ref{filtration} to the last term above gives
\begin{align*}
	\len{R}{\dfrac{R}{\unideal{f+\epsilon} + g'R + J^{\fbp{p^e}}}} \le p^{e(a-1)} \len{R}{\dfrac{R}{\unideal{f+\epsilon} + g'R + \unideal{h}^{\fbp{p^{e_0}}}}} \, .
\end{align*}
Moreover, equality (\ref{ghIideals}), established earlier, says that $$\unideal{f+\epsilon} + g'R + \unideal{h}^{\fbp{p^{e_0}}}R = I + gR + \unideal{h}^{\fbp{p^{e_0}}}R.$$  Putting all of this together, we get
\begin{align*}
	\len{R}{\dfrac{R}{\unideal{f+\epsilon} + g'R + J^{\fbp{p^e}}}} \le p^{e(a-1)} \len{R}{\dfrac{R}{I + gR + \unideal{h}^{\fbp{p^{e_0}}}}}.
\end{align*}
\indt Now, set
\begin{align*}
	C = C_{e_0}:= p^a m \len{R}{\dfrac{R}{I + gR + \unideal{h}^{\fbp{p^{e_0}}}}}, 
\end{align*}
and recall that $v_2 = \kpdeg p^a \le \kpdeg p^a m.$  We have shown
\begin{align*}
\len{R}{\dfrac{Q_i}{J^{\fbp{p^e}}Q_i}} \le v_i \len{R}{\dfrac{R}{\unideal{f+\epsilon} + g'R + J^{\fbp{p^e}}}} \le \kpdeg C p^{e(a-1)},
\end{align*}
for $i = 1, 2.$  Put this on both sides of the compound inequality, (\ref{ineqQ}) above, and divide by $\kpdeg.$  We arrive at the desired inequality,
\begin{align*}
	\Biggl|\len{R}{\dfrac{R}{\unideal{f+\epsilon} + J^{\fbp{p^{e+1}}}}} - p^a\len{R}{\dfrac{R}{\unideal{f+\epsilon} + J^{\fbp{p^{e}}}}} \Biggr| \le Cp^{e(a-1)} \, .
\end{align*}
\end{proof}
\begin{cor} \label{HkGrowth}
Let $(R, \mf{m}_R, \kappa)$ be an $F$-finite local ring, and suppose that $I \subset R$ is an ideal such that $a = \dim R/I \ge 1.$  Assume that conditions (i), (ii), and (iii) of proposition \ref{UniformBounds} hold. Given any $e_0 \ge 0,$ let $C = C_{e_0}$ denote the corresponding constant, and $T \in \mathbb{N}$ the corresponding integer shown to exist in proposition \ref{UniformBounds}, above.  Then, for all $J \subset R$ with $\mf{m}_R^{\fbp{p^e}} \subset J,$ and all $e \ge 1,$
\begin{align*}
	\Biggl|\ehk{J\olper{R}} - \dfrac{1}{p^{ea}}\len{R}{\dfrac{R}{\unideal{f+\epsilon} + J^{\fbp{p^{e}}}}} \Biggr| \le Cp^{-e} 
\end{align*}
for all minimal generators $(f_1, \dots, f_c) = I,$ and every $\epsilon_1, \dots, \epsilon_c \in \mf{m}_R^T.$
\end{cor}
\begin{proof}
Fix $e \ge 1,$ and for any $k\ge 0$ let $\ell_{k}$ denote the length
\begin{align*}
	\ell_k : = \len{R}{\dfrac{R}{\unideal{f+\epsilon} + J^{\fbp{p^{k}}}}} \, .
\end{align*}
Then, for any positive integer $e' \ge 1,$ we have
\begin{align*}
	\Biggl|p^{e'a}&\len{R}{\dfrac{R}{\unideal{f+\epsilon} + J^{\fbp{p^{e}}}}} - \len{R}{\dfrac{R}{\unideal{f+\epsilon} + J^{\fbp{p^{e+e'}}}}} \Biggr| \\
	&= \left|p^{e'a}\ell_e - \ell_{e+e'} \right| \\
	&= \left|p^{e'a}\ell_e - p^{(e'-1)a}\ell_{e+1} + p^{(e'-1)a}\ell_{e+1} - p^{(e'-2)a}\ell_{e+2} + \dots p^a \ell_{e+e'-1} - \ell_{e+e'} \right| \\
	&\le p^{(e'-1)a} \Bigl|p^a \ell_e - \ell_{e+1}\Bigr| + p^{(e'-2)a}\Bigl|p^a \ell_{e+1} - \ell_{e+2}\Bigr| + \dots + \Bigl|p^a \ell_{e+e'-1} - \ell_{e+e'}\Bigr| \, .
\end{align*}
Proposition \ref{UniformBounds} gives a bound on each of these terms,
\begin{align*}
	\left|p^a\ell_{k} - \ell_{k+1}\right| \le Cp^{k(a-1)},
\end{align*}
and therefore we have,
\begin{align*}
	\Biggl|p^{e'a}&\len{R}{\dfrac{R}{\unideal{f+\epsilon} + J^{\fbp{p^{e}}}}} - \len{R}{\dfrac{R}{\unideal{f+\epsilon} + J^{\fbp{p^{e+e'}}}}} \Biggr| \\
		&\le  Cp^{(e'-1)a + e(a-1)} + Cp^{(e'-2)+(e+1)(a-1)} + \dots + Cp^{(e+e'-1)(a-1)} \\
		&\le \left(\dfrac{1-\dfrac{1}{p^{e'}}}{(p-1)p^{a-1}}\right)Cp^{-e}p^{(e+e')a} \\
		&\le Cp^{-e}p^{(e+e')a} \, .
\end{align*}
\indt Dividing by $p^{(e+e')a}$ we conclude that, for every $e' \ge 0,$ 
\begin{align*}
	\Biggl|\dfrac{1}{p^{ea}}\len{R}{\dfrac{R}{\unideal{f+\epsilon} + J^{\fbp{p^{e}}}}} - \dfrac{1}{p^{(e+e')a}}\len{R}{\dfrac{R}{\unideal{f+\epsilon} + J^{\fbp{p^{e+e'}}}}} \Biggr| \le Cp^{-e} \, .
\end{align*}
Now, taking the limit as $e' \to \infty$ establishes the claim.
\end{proof}
\indt This brings us to the extension of theorem \ref{IllyaThomasHK}. 
\begin{thm} \label{HKThm}
	Let $(R, \mf{m}_R, \kappa)$ be an $F$-finite local ring, and suppose that $I \subset R$ is an ideal.  If $\dim R/I \ge 1,$ then assume that conditions (i), (ii), and (iii) of proposition \ref{UniformBounds} hold.  Let $J \subset R$ be an $\mf{m}_R$-primary ideal. \\
\indt Then, for any $\delta > 0,$ there is a $T \in \mathbb{N}$ such that for all minimal generators $(f_1, \dots, f_c) = I$ and all $\epsilon_1, \dots, \epsilon_c \in \mf{m}_R^T$
\begin{align*}
	\Biggl|\ehk{J\ol{R}} - \ehk{J\olper{R}} \Biggr| < \delta \, .
\end{align*}  
\end{thm}
\begin{proof}
If $\dim R/I = 0,$ then $I$ is $\mf{m}_R$-primary and there is a $T$ such that $\ol{R} = \olper{R}$ for all minimal generators $\unideal{f} = I$ and all $\epsilon_1, \dots, \epsilon_c \in \mf{m}_R^T.$ \\
\indt Now assume that $\dim R/I \ge 1,$ and the conditions of proposition \ref{UniformBounds} are satisfied.  By corollary \ref{HkGrowth} and proposition \ref{UniformBounds} there is a constant $C$ (depending on $J$), and $T$ such that, for all $e,$ all minimal generators $\unideal{f} = I$ and all $\epsilon_1, \dots, \epsilon_c \in \mf{m}_R^T,$ 
\begin{align*}
	\Biggl|\ehk{J\olper{R}} - \dfrac{1}{p^{ea}}\len{R}{\dfrac{R}{\unideal{f+\epsilon} + J^{\fbp{p^{e}}}}} \Biggr| \,\, + \,\, \Biggl|\ehk{J\ol{R}} - \dfrac{1}{p^{ea}}\len{R}{\dfrac{R}{I + J^{\fbp{p^{e}}}}} \Biggr| \le 2Cp^{-e} \, .
\end{align*} 
Fix an $e$ such that this bound satisfies $2Cp^{-e} < \delta.$  Hence, $J^{\fbp{p^e}}$ is a fixed $\mf{m}_R$-primary ideal and by lemma \ref{naklemma}, after possibly increasing $T$, we may assume that $$ \unideal{f+\epsilon} + J^{\fbp{p^e}} = I + J^{\fbp{p^e}},$$ for all the corresponding $\unideal{f+\epsilon}.$ \\
\indt Putting this together, we have,
\begin{align*}
	\Biggl|\ehk{J\olper{R}}& - \ehk{J\ol{R}}  \Biggr| \\
	&= \Biggl|\ehk{J\olper{R}} -  \dfrac{1}{p^{ea}}\len{R}{\dfrac{R}{\unideal{f+\epsilon} + J^{\fbp{p^e}}}} + \dfrac{1}{p^{ea}}\len{R}{\dfrac{R}{I + J^{\fbp{p^e}}}} - \ehk{J\ol{R}} \Biggr|	\\
	&\le \Biggl|\ehk{J\olper{R}} - \dfrac{1}{p^{ea}}\len{R}{\dfrac{R}{\unideal{f+\epsilon} + J^{\fbp{p^{e}}}}} \Biggr| \,\, + \,\, \Biggl|\ehk{J\ol{R}} - \dfrac{1}{p^{ea}}\len{R}{\dfrac{R}{I + J^{\fbp{p^{e}}}}} \Biggr| \\
	&< \delta \, .
\end{align*} 
\end{proof}
\indt Before moving on to discuss some interesting examples, we mention a quick application of our results. When theorem \ref{HKThm} applies, the Hilbert-Kunz multiplicity of $R/I$ can be realized as the limit of Hilbert-Kunz multiplicities of finite radical extensions of $R,$ obtained by joining the roots of minimal generators of $I$.
\begin{thm} \label{HKlimit}
	Let $(R, \mf{m}_R, \kappa)$ be a reduced, $F$-finite local ring, and suppose that $I \subset R$ is an ideal.  Assume that the conditions of theorem \ref{HKThm} are satisfied, and let $$(f_1, \dots, f_c) = I$$ be any minimal generating set.  Then,
	\begin{align*}
		\lim_{n_1 \to \infty, \dots, n_c \to \infty} \ehk{R[f_1^{1/n_1}, \dots, f_c^{1/n_c}]} = \ehk{R/I} \, .
	\end{align*}
\end{thm}
\begin{proof}
	Assume $R$ is complete.  The extension $R \xhookrightarrow{} R',$ with $R' = R[[X_1, \dots, X_c]]$ the power series ring in $X_1, \dots, X_c$, is faithfully flat and conditions (i), (ii) and (iii) continue to hold for $IR' \subset R'.$  Therefore, theorem \ref{HKThm} gives
\begin{align*}
	\lim_{n_1 \to \infty, \dots, n_c \to \infty} \ehk{R[f_1^{1/n_1}, \dots, f_c^{1/n_c}]} = \lim_{n_1 \to \infty, \dots, n_c \to \infty} \ehk{R'/(f_1-X_1^{n_1}, \dots, f_c-X_c^{n_c})} = \ehk{R'/IR'}.
\end{align*}
\indt The induced extension, $$R/I \xhookrightarrow{} R'/IR' \cong \left(R/I\right)[[X_1, \dots, X_c]]$$ is faithfully flat and the fiber, $\kappa[[X_1, \dots, X_c]],$ is regular.  By a standard result (see proposition 3.9 of \cite{kunz1976noetherian}), we have
\begin{align*}
	\ehk{R'/IR'} = \ehk{R/I}.
\end{align*}
\end{proof}
\section{Examples and Further Directions}
\indt We finish with some interesting examples and questions that we would like to address in future research.
\subsection{Examples} \label{examples}
\indt Given $I = \unideal{f} \subset R,$ such that the conditions of theorem \ref{HKThm} are satisfied, and an $\mf{m}_R$-primary ideal $J \subset R$, it is natural to wonder if the Hilbert-Kunz multiplicities $\ehk{\olper{R}}$ are actually equal for $T \gg 0.$  The following example, which features as example 4.1 in \cite{polstra_smirnov_2018}, shows we cannot expect this to happen in general.
\begin{emple} \label{notiso}
	(Example 4.1 of \cite{polstra_smirnov_2018})  Let $f = xy \in R = k[[x, y, t]],$ with $k$ a field of characteristic $p > 0.$ The associativity formula gives
	\begin{align*}
		\ehk{R/(f)} = \ehk{R/(x)} \len{R}{\left(R/(xy)\right)_{(x)}} + \ehk{R/(y)} \len{R}{\left(R/(xy)\right)_{(y)}} = 2,
	\end{align*}
	while it is known that (see theorem 3.1 of \cite{conca1996hilbert}, or example 3.18 of \cite{huneke2014hilbert}),
	\begin{align*}
	 	\ehk{R/(f+t^n)} = 2 - \dfrac{1}{n} \, .
	\end{align*} 
\end{emple}
\indt It is interesting to observe what goes wrong in examples where the multiplicities behave badly with respect to $\mf{m}$-adic perturbations.
\begin{emple} \label{dimfail1}
	 (Example 4.3 of \cite{polstra_smirnov_2018}) Let $k$ be a field of positive characteristic, and set $R = k[[x, y, t]]/(xy, xt).$  The minimal primes of $R$ are $(x)$ and $(y, t)$ --- note that $R$ is not equidimensional.  $R/(y)$ has dimension $1$, so we have
	 \begin{align*}
	 	\ehk{R/(y)} = e\left(R/(y)\right) = e\left(k[[x, t]]/(xt)\right) = 2, 
	 \end{align*} 
	 while, for every $n \ge 1,$ $(x)$ is the only minimal prime of $R/(y+x^n) \cong k[[x, t]]/(x^{n+1}, xt),$ so that
	 \begin{align*}
	 	\ehk{R/(y+x^n)} = e\left(R/(y+x^n)\right) = e\left(R/(x, y)\right) \len{R}{\left(R/(y+x^n)\right)_{(x)}} = e\left(k[[t]]\right) \len{R}{\left(k[[x, t]]/(x^{n+1}, xt)\right)_{(x)}} = 1
	 \end{align*}
\indt Observe that $\hgy{1}{y ; R} = \col{0}{y}{R} = (x),$ so that $\ann{R}{\hgy{1}{y; R}} = (y, t).$  In particular, the condition on the dimension of $\hgy{1}{y; R}$ appearing in theorems \ref{HSper} and \ref{HKThm} is not satisfied: 
\begin{align*}
	\dim R/\ann{R}{\hgy{1}{y; R}} = \dim R/(y, t) = 1 = \dim R/(y) \, .
\end{align*}
\indt We should point out that, in this example, $(y,t)$ and $(x, y)$ are the minimal primes of $(y),$ and $$R_{(y,t)} \cong \dfrac{k[[x, y, t]]_{(y, t)}}{(y,t)k[[x, y, t]]_{(y, t)}}$$ and $$R_{(x,y)} \cong \dfrac{k[[x, y, t]]_{(x, y)}}{(x)k[[x, y, t]]_{(x, y)}}$$ are both certainly Cohen-Macaulay.  However, while $y$ is a parameter on $R$, $y$ clearly does not form a regular sequence on $R_{(y,t)}.$   \par 
\indt If instead we consider perturbations of $(x) \subset R$ the situation is quite pathological.  Unlike $y$, $x$ is not a parameter on $R,$ so the perturbations may not even have the same dimension.  It is not difficult to verify directly that the condition on $\mathbf{H}_1$ fails here too: $\hgy{1}{x ; R} = \col{0}{x}{R} = (y,t),$ so that $\ann{R}{\hgy{1}{x; R}} = (x),$ and $$ R/\ann{R}{\hgy{1}{x; R}} = R/(x).$$  In this case, $R/(x) \cong k[[y, t]]$ is regular (of dimension $2$) so 
\begin{align*}
	e \left(R/(x)\right) = \ehk{R/(x)} = 1,
\end{align*} 
while an easy calculation shows that, for any $n,$ $$R/(x+t^n) \cong \dfrac{k[[y, t]]}{(t^{n+1}, t^ny)}.$$  Each of these $R/(x+t^n)$ is $1$ dimensional with a unique minimal prime, $(t),$ and we get 
\begin{align*}
	\ehk{R/(x+t^n)} = e\left(R/(x+t^n)\right) &= e \left(\dfrac{k[[y,t]]}{(t)}\right) \len{R}{\left(\dfrac{k[[y, t]]}{(t^{n+1}, t^ny)}\right)_{(t)}} \\
	&= \len{R}{\dfrac{k[[y, t]]_{(t)}}{(t^n)k[[y, t]]_{(t)}}} = n \, .
\end{align*}
\end{emple} \\
\begin{emple} \label{dimfail2}
 In proposition 4.4 of \cite{polstra_smirnov_2018} the authors work out the following example, reportedly due to Hochster.
\begin{prop} (Proposition 4.4 of \cite{polstra_smirnov_2018})  Let $k$ be a field of characteristic $p > 0$ and let $S$ be the following subring of the power series ring in $x, y$ and $z$ over $k$: $$S := k[[x^3, x^2y, y^3, y^2z, z^3, z^2x]] \subset k[[x, y, z]] \, .$$ \\
Let $R := S/(x^3).$ \\
\indt Then we have,
\begin{itemize}
\item[(i)] $\ehk{R/(y^3)} = e \left(R/(y^3)\right) = 11$;
\item[(ii)] For every $k \not \equiv 1 \mod 3,$ $$\ehk{R/(y^3+z^{3k})} = e \left(R/(y^3+z^{3k})\right) \neq 11 \, .$$
\end{itemize}
\end{prop}
\vspace{.1in}
\indt We have $$\hgy{1}{y^3; R} = \col{0}{y^3}{R}.$$  As noted in \cite{polstra_smirnov_2018}, $z^3$ is a parameter on $R/(y^3),$ which has dimension $1$.  Observe that, in $R$, $z^{3n}y^3 \neq 0$ for any $n$ --- in particular, $z^3 \not \in \sqrt{\ann{R}{\hgy{1}{y^3; R}}}.$ Therefore $\dim R/\ann{R}{\hgy{1}{y^3; R}} \ge 1.$  Since $R/(y^3)$ has dimension $1$, we conclude that
\begin{align*}
	\dim R/\ann{R}{\hgy{1}{y^3; R}} = \dim R/(y^3) = 1
\end{align*} 
and, once again we see that $\mathbf{H}_1$ does not satisfy the condition of theorems \ref{HSper} and \ref{HKThm}. \\
\indt In \cite{polstra_smirnov_2018} the authors show that the radical of the ideal we are perturbing, $$\sqrt{y^3R} = \left(x^2y, y^3, y^2z, z^2x\right)R,$$ is a prime ideal.  In particular, condition (iii) of theorem \ref{HKThm} also fails to hold for this example.
\end{emple}
\subsection{Further Directions}
\begin{question}
When $\unideal{f} = I \subset R$ satisfies the conditions of theorem \ref{HSper}, there are equalities
\begin{align*}
	\sum_{\mf{p} \in \minh{R/I}} e\left(R/\mf{p}\right)\len{R}{\left(R/I\right)_{\mf{p}}} = \sum_{\mf{q} \in \minh{R/\unideal{f+\epsilon}}} e\left(R/\mf{q}\right)\len{R}{\left(R/\unideal{f+\epsilon}\right)_{\mf{q}}}
\end{align*}
for all sufficiently small perturbations, $\unideal{f+\epsilon},$ of $I$.  What can be said about the behavior of the sets $\minh{\unideal{f+\epsilon}} \subset \spec{R}$ as we range over small perturbations of $I$?  The equality above certainly constrains the size of these sets.  Note that in example \ref{notiso} the perturbations described are all prime and all distinct --- in particular, the number of distinct primes appearing the sets $\minh{\unideal{f+\epsilon}}$ can be infinite.  \\
\end{question}
\begin{question}
 The equidimensionality and generic reducedness assumptions are only used in the proof of proposition \ref{HkGrowth} to invoke the Cohen-Gabber structure theorem, so that we have a generically \'etale Cohen extension to work with.  Are there any obvious reasons to expect these conditions to have anything to do with the conclusion?   Is there some way to drop them? 
\end{question}
\begin{question}
How fundamental is the generic vanishing of $\mathbf{H}_1(I)$ to the behavior of perturbations if $I$?  For example, is theorem \ref{HSper} the best possible?  Another way to ask this question is: is the $\mathbf{H}_1$ vanishing condition introduced in section \ref{techcon} the 'correct' general condition on $I \subset R$?  
\end{question}
\begin{question}
We made an effort to formulate lemmas \ref{dLem} and \ref{mainlemma} in more general form than needed for this specific project; what other applications do these results have?  
\end{question}
\begin{question}
It is known that the $F$-singularities of $\mathbb{Q}$-gorenstein rings behave particularly well under $\mf{m}$-adic perturbations (e.g. see \cite{AIstab} and \cite{polstra2020fpurity}).  How do the test ideals, Cartier algebras, and other constructions from $F$-singularity theory behave under perturbations in the $\mathbb{Q}$-gorenstein (or Gorenstein) case? 
\end{question}
\begin{question}
In \cite{taylor2018interpolating} William Taylor has introduced a multiplicity function that interpolates between the Hilbert-Samuel and Hilbert-Kunz multiplicities.  It is very natural to wonder if our methods might be extended to get results analogous to theorem \ref{HSper} or theorem \ref{HKThm} for s-multiplicity.
\end{question}
\section{Bibliography}
\bibliography{Pert_NLR}
\bibliographystyle{alpha}
\end{document}